\documentclass[12pt,reqno]{amsart}

\usepackage{multirow}
\usepackage{hhline}
\usepackage{float}
\usepackage{multirow}
\usepackage{multicol}
\usepackage{mathtools}
\usepackage{times}
\usepackage[T1]{fontenc}
\usepackage{mathrsfs}
\usepackage{latexsym}
\usepackage[dvips]{graphics}
\usepackage[titletoc, title]{appendix}
\usepackage{amsmath,amsfonts,amsthm,amssymb,amscd}
\usepackage[dvipsnames]{xcolor}
\usepackage{hyperref}
\usepackage{amsmath}
\usepackage[utf8]{inputenc}

\usepackage{color}
\usepackage{breakurl}

\usepackage{comment}
\newcommand{\bburl}[1]{\textcolor{blue}{\url{#1}}}

\usepackage{caption}

\newtheorem{thm}{Theorem}[section]

\newtheorem{lem}[thm]{Lemma}
\newtheorem{prop}[thm]{Proposition}

\newtheorem{defi}[thm]{Definition}
\newtheorem{rek}[thm]{Remark}

\usepackage[utf8]{inputenc}

\DeclareFixedFont{\ttb}{T1}{txtt}{bx}{n}{12} % for bold
\DeclareFixedFont{\ttm}{T1}{txtt}{m}{n}{12}  % for normal

\usepackage{color}
\definecolor{deepblue}{rgb}{0,0,0.5}
\definecolor{deepred}{rgb}{0.6,0,0}
\definecolor{deepgreen}{rgb}{0,0.5,0}

\usepackage{listings}

\newcommand\pythonstyle{\lstset{
language=Python,
basicstyle=\ttm,
morekeywords={self},              % Add keywords here
keywordstyle=\ttb\color{deepblue},
emph={MyClass,__init__},          % Custom highlighting
emphstyle=\ttb\color{deepred},    % Custom highlighting style
stringstyle=\color{deepgreen},
frame=tb,                         % Any extra options here
showstringspaces=false
}}

\lstnewenvironment{python}[1][]
{
\pythonstyle
\lstset{#1}
}
{}

\newcommand\pythoninline[1]{{\pythonstyle\lstinline!#1!}}

\definecolor{ao}{rgb}{0.0, 0.5, 0.0}

\numberwithin{equation}{section}

\DeclareFontFamily{U}{mathx}{}
\DeclareFontShape{U}{mathx}{m}{n}{<-> mathx10}{}
\DeclareSymbolFont{mathx}{U}{mathx}{m}{n}
\DeclareMathAccent{\widehat}{0}{mathx}{"70}
\DeclareMathAccent{\widecheck}{0}{mathx}{"71}

\begin{document}

\title{Schreier-Type Sets and Linear Recurrences:\\ Connections and Developments}

\author[H. V. Chu]{H\`ung Vi\d{\^e}t Chu}
\email{\textcolor{blue}{\href{mailto:hchu@wlu.edu}{hchu@wlu.edu}}}
\address{Department of Mathematics\\ Washington and Lee University, Lexington, VA 24450, USA}

\thanks{}



\maketitle
 
\begin{abstract}
We demonstrate several common techniques for proving linear recurrences from counting Schreier-type sets. These techniques include formula-based arguments, bijective proofs, mathematical induction, the inclusion-exclusion principle, and the characteristic polynomial method.

As new contributions, we examine symmetric maximal Schreier sets, Schreier sets that contain a prescribed integer, and Schreier sets that avoid integers belonging to a fixed arithmetic progression. Along the way, we employ useful techniques for
identifying meaningful patterns in data and establishing technical identities. The results presented here, together with the diverse proof techniques employed, are expected to serve as a valuable resource for undergraduate researchers interested in this area.
\end{abstract}

\tableofcontents

%%%%%%%%%%%%%%%%%%%%%%%%%%%%%%%%%%%%%%%%%%%%%%%%%%%%%%%%%%%%%%%%%%%%%%%%%%%%%%%%%%%%%%%%%%%%%%%%%%%%%%%%%%%%%%%%%%%%%%%%%%%%%%%%%%%%%%%%%%%%%%%%%%%%%%%%%%%%%%%%%%%%%%%%%%%%%%%%%%%%%%%%%%%%%%%%%%%%%%%%%%%%%%%%%%%%%%%%%%%%%%%%%%%%%%%%%%%%%%%%%%%%%%%%%%%%%%%%%%%%%%%%%%%%%%%%%%%%%%%%%%%%%%%%%%%%%%%%%%%%%%%%%%%%%%%%%%%%%%%%%%%%%%%%%%%%%%%%%%%%%%%%%%%%%%%%%%%%%%%%%%%

\section{Introduction}\label{sec-intro}

Let $A\subset \mathbb{N}$ be nonempty and finite. Then $A$ is \textit{Schreier} (\textit{maximal Schreier}, respectively) if $\min A\ge |A|$ ($\min A = |A|$, respectively). The Fibonacci numbers are defined recursively as $F_1 = F_2 = 1$ and $F_n = F_{n-1} + F_{n-2}$ for $n\ge 3$. For each positive integer $n$, we use $[n]$ to denote the set $\{1, 2, \ldots, n\}$. In a blog post \cite{Bi}, Bird observed that
\begin{equation}\label{e-1}|\{F\subset [n]\,:\, F\mbox{ is Schreier and }n\in F\}|\ =\ F_n.\end{equation}
In Section \ref{formula-intro}, we show two different ways to prove \eqref{e-1}: one is formula-based, and the other due to Bird involves bijective maps. In Section \ref{casestudy-symmetric-maximal}, we use these tools more intensively as we count symmetric maximal Schreier subsets $F$ of $\{1, 2, \ldots, 2n+1\}$ with $2n+1\in F$. We recall the definition of symmetric sets. Given a number $z$ and a nonempty set $A$, we let $z + A := \{z+a:a\in A\}$, $zA:=\{za: a\in A\}$, and $s_A:= (\min A + \max A)/2$.
We say $A$ is \textit{symmetric} if 
\begin{equation}\label{e1}A \ =\ 2s_A - A,\end{equation} 
which is equivalent to
\begin{equation}\label{e2}s_A-\left\{a\in A\,:\, a < s_A\right\} \ =\ \left\{a\in A\,:\, a > s_A\right\}-s_A.\end{equation}
While we later use \eqref{e1} to check for symmetry, one may find \eqref{e2} more intuitive as \eqref{e2} states that $A$ is symmetric about $s_A$. Appendix \ref{symset} gives a proof of the equivalence. Let 
$$\mathcal{M}_n\ :=\ \{F\subset [n]\,:\, F\mbox{ is symmetric maximal Schreier and }n\in F\}.$$
Here is our first result.
\begin{thm}\label{m1}
    For $n\ge 0$, we have
    \begin{equation}\label{e70}|\mathcal{M}_{2n+1}|\ =\ F_{n+1}.\end{equation}
\end{thm}
For the bijective proof of \eqref{e70}, we show not only the maps but also how the data are used to figure out these maps.

Next, we count Schreier sets that contain a fixed integer $u\ge 1$. Let 
$$\mathcal{F}_{u,n}\ :=\ \{F\subset [n]\,:\,F\mbox{ is Schreier and }u\in F\}.$$
\begin{table}[H]
\centering
\begin{tabular}{ |c| c| c| c| c| c| c| c| c| c| c| c| c| c| c| c| c| c| c| c| c|}
\hline
$n$ &$1$& $2$ & $3$ & $4$ & $5$ & $6$ & $7$ & $8$ & $9$ & $10$ & $11$ & $12$ & $13$ & $14$ & $15$ \\
\hline
$|\mathcal{F}_{1,n}|$ & $1$ & $1$ & $1$ & $1$ & $1$ & $1$ & $1$ & $1$ & $1$ & $1$ & $1$ & $1$ & $1$ & $1$ & $1$  \\
\hline
$|\mathcal{F}_{2,n}|$ &$0$&  $1$ & $2$ & $3$ & $4$ & $5$ & $6$ & $7$ & $8$ & $9$ & $10$ & $11$ & $12$ & $13$ & $14$  \\
\hline
$|\mathcal{F}_{3,n}|$ &$0$& $0$ & $2$ & $3$ & $5$ & $8$ & $12$ & $17$ & $23$ & $30$ & $38$ & $47$ & $57$ & $68$ & $80$   \\
\hline
$|\mathcal{F}_{4,n}|$  & $0$ & $0$ & $0$& $3$ & $5$ & $8$ & $13$ & $21$ & $33$ & $50$ & $73$ & $103$ & $141$ & $188$ & $245$  \\
\hline
$|\mathcal{F}_{5,n}|$ & $0$ & $0$ & $0$& $0$ & $5$ & $8$ & $13$ & $21$ & $34$ & $55$ & $88$ & $138$ & $211$ & $314$ & $455$  \\
\hline
\end{tabular}
\caption{The first $15$ values of $(|\mathcal{F}_{u, n}|)_{n=1}^\infty$ with $1\le u\le 5$.}
\label{Data_oneint}
\end{table}

At first glance, one may find it challenging to detect a recurrence relation for later sequences in Table \ref{Data_oneint}. Section \ref{sec-nointeger} demonstrates how the data suggest not only Recurrence \eqref{e40} below but also two different proofs: one based on the inclusion-exclusion principle and the other based on mathematical induction.

\begin{thm}\label{m2}
Let $u\in \mathbb{N}$. For $u\le n\le 2u-1$, we have
$|\mathcal{F}_{u,n}| = F_{n}$.
    For $n\ge 2u$,
    \begin{equation}\label{e40}|\mathcal{F}_{u,n}|\ =\ \sum_{i=1}^u (-1)^{i-1}\binom{u}{i}|\mathcal{F}_{u,n-i}|.\end{equation}
\end{thm}

Finally, suppose we want to prove a complicated linear recurrence $R_1$ for a given sequence $(a_n)$, which is a periodic subsequence of a sequence $(b_n)$ satisfying a much simpler linear recurrence $R_2$. The characteristic polynomial method enables us to exploit this relationship to establish $R_1$. Here is the outline of the method:
$$\mbox{Goal: prove that a sequence }(a_n)\mbox{ satisfies the linear recurrence }R_1.$$
\begin{itemize}
    \item Step 1: Find a sequence $(b_n)$ such that $(a_n)$ is an $s$-periodic subsequence of $(b_n)$; that is, $(a_n)$ is obtained by taking every $s$th term of $(b_n)$. Moreover, $(b_n)$ is known to satisfy a linear recurrence $R_2$ that is simpler than $R_1$. 
    \item Step 2: The recurrences $R_1$ and $R_2$ are encoded by the polynomials $q(x)$ and $p(x)$, respectively. If we can verify that $p(x)$ divides $q(x^s)$, then $(b_n)$ also satisfies the recurrence encoded by $q(x^s)$. Since $(a_n)$ is an $s$-periodic subsequence of $(b_n)$, it follows that $(a_n)$ satisfies $q(x)$ and thus, $R_1$.
\end{itemize}

We illustrate this method by generalizing \cite[Theorem 3]{CGKMTV}, which concerns Schreier sets avoiding multiples of a fixed integer, to the broader setting of Schreier sets that avoid terms of a fixed arithmetic progression. 

For $a,b\in \mathbb{N}$ with $b\ge 2$ and $0\le a < b$, let 
$G^{(a,b)}_n$ be the set of the first $n$ positive integers that are not $a$ modulo $b$. Let $g^{(a,b)}(n)$ denote the $n$th positive integer that is not $a$ modulo $b$ so that 
$$G^{(a,b)}_n\ =\ \{g^{(a,b)}(1), g^{(a,b)}(2), \ldots, g^{(a,b)}(n)\}.$$
Then
\begin{equation}\label{e32}g^{(a,b)}(n)\ =\ \begin{cases}\left\lfloor \frac{bn-1}{b-1}\right\rfloor, &\mbox{ if }a = 0,\\ \left\lfloor \frac{bn-a}{b-1}\right\rfloor+1, &\mbox{ if } 1\le a\le b-1.\end{cases}\end{equation}
Interested readers may see \cite[Lemma 2]{CGKMTV} for the proof of \eqref{e32} when $a = 0$. To prove the formula for $a\ge 1$, we write $n = (b-1)j+m$ with $j\ge 0$ and $0\le m\le b-2$. We have
\begin{align*}
    \left\lfloor \frac{bn-a}{b-1}\right\rfloor+1&\ =\ \left\lfloor \frac{b((b-1)j+m)-a}{b-1}\right\rfloor+1\\
    &\ =\ bj + 1 + m + \left\lfloor \frac{m-a}{b-1}\right\rfloor\\
    &\ =\ \begin{cases} bj+m, &\mbox{ if }0\le m \le a-1,\\ bj+1+m, &\mbox{ if }a\le m\le b-2.\end{cases}
\end{align*}
This confirms \eqref{e32} in the case $a\ge 1$.

For $n\in \mathbb{N}$, define 
$$\mathcal{D}^{(a,b)}_{n}\ :=\ \left\{F\subset G^{(a,b)}_{n}\,:\, g^{(a,b)}(n)\in F\mbox{ and }F\mbox{ is Schreier}\right\}.$$

\begin{table}[H]
\centering
\begin{tabular}{ |c| c| c| c| c| c| c| c| c| c| c| c| c| c| c| c| c| c| c|}
\hline
$n$ &$1$& $2$ & $3$ & $4$ & $5$ & $6$ & $7$ & $8$ & $9$ & $10$ & $11$ & $12$ & $13$  \\
\hline
$|\mathcal{D}^{(0, 2)}_{n}|$ & $1$ & $1$ & $2$ & $4$ & $7$ & $12$ & $21$ & $37$ & $65$ & $114$ & $200$ & $351$ & $616$  \\
\hline
$|\mathcal{D}^{(1, 2)}_{n}|$ & $1$ & $2$ & $3$ & $5$ & $9$ & $16$ & $28$ & $49$ & $86$ & $151$ & $265$ & $465$ & $816$  \\
\hline
$|\mathcal{D}^{(0, 3)}_n|$ & $1$ & $1$ & $2$ & $3$ & $5$ & $9$ & $16$ & $28$ & $48$ & $81$ & $136$ & $229$ & $388$  \\
\hline
$|\mathcal{D}^{(1, 3)}_n|$ & $1$ & $2$ & $3$ & $5$ & $8$ & $13$ & $22$ & $38$ & $66$ & $114$ & $195$ & $331$ & $560$ \\
\hline
$|\mathcal{D}^{(2, 3)}_n|$ & $1$ & $1$ & $2$ & $4$ & $7$ & $12$ & $20$ & $33$ & $55$ & $93$ & $159$ & $273$ & $468$ \\
\hline
$|\mathcal{D}^{(0, 4)}_n|$ & $1$ & $1$ & $2$ & $3$ & $5$ & $8$ & $13$ & $22$ & $38$ & $66$ & $114$ & $195$ & $330$ \\
\hline
$|\mathcal{D}^{(1, 4)}_n|$ & $1$ & $2$ & $3$ & $5$ & $8$ & $13$ & $21$ & $34$ & $56$ & $94$ & $160$ & $274$ & $469$  \\
\hline
$|\mathcal{D}^{(2, 4)}_n|$ & $1$ & $1$ & $2$ & $4$ & $7$ & $12$ & $20$ & $33$ & $54$ & $88$ & $144$ & $238$ & $398$  \\
\hline
$|\mathcal{D}^{(3, 4)}_n|$ & $1$ & $1$ & $2$ & $3$ & $5$ & $9$ & $16$ & $28$ & $48$ & $81$ & $135$ & $223$ & $367$ \\
\hline
\end{tabular}
\caption{The first $13$ values of $(|\mathcal{D}^{(a,b)}_{n}|)_{n=1}^\infty$ with $2\le b\le 4$ and $0\le a\le b-1$.}
\label{Dn}
\end{table}

Since \cite[Theorem 3]{CGKMTV} already handled the case $a = 0$, we assume $a\ge 1$ and prove the following theorem in Section \ref{sec-amodb}. It turns out that $(|\mathcal{D}^{(a,b)}_n|)_{n=1}^\infty$ satisfies the same recurrence that is independent of $a$, differing only in initial values; however, the proof is considerably more technical.

\begin{thm}\label{mt3}
For $n\le 2b-1$, 
$$|\mathcal{D}^{(a,b)}_n|\ =\ \sum_{k=0}^{\left\lfloor \frac{nb-a}{2b-1}\right\rfloor} \binom{n-k+\left\lfloor\frac{k-a}{b}\right\rfloor}{k}.$$
    For $n\ge 2b$,
    $$|\mathcal{D}^{(a,b)}_n|\ =\ \sum_{i=1}^b(-1)^{i-1}\binom{b}{i}|\mathcal{D}^{(a,b)}_{n-i}| + |\mathcal{D}^{(a,b)}_{n-2b+1}|.$$
\end{thm}

We close the introduction with some notation to be used throughout the paper: given $k\in \mathbb{N}$, if $\mathcal{A}$ is a collection of sets, then $\mathcal{A}(k)$ denotes the collection of all sets in $\mathcal{A}$ of size $k$.

%%%%%%%%%%%%%%%%%%%%%%%%%%%%%%%%%%%%%%%%%%%%%%%%%%%%%%%%%%%%%%%%%%%%%%%%%%%%%%%%
%%%%%%%%%%%%%%%%%%%%%%%%%%%%%%%%%%%%%%%%%%%%%%%%%%%%%%%%%%%%%%%%%%%%%%%%%%%%%%%%%
%%%%%%%%%%%%%%%%%%%%%%%%%%%%%%%%%%%%%%%%%%%%%%%%%%%%%%%%%%%%%%%%%%%%%%%%%%%%%%%%%
%%%%%%%%%%%%%%%%%%%%%%%%%%%%%%%%%%%%%%%%%%%%%%%%%%%%%%%%%%%%%%%%%%%%%%%%%%%%%%%%%

\section{Formula-based arguments versus bijective proofs}\label{formula-intro}

In this section, we present both formula-based and bijective proofs of \eqref{e-1} and discuss the advantages and disadvantages of each approach.
Recall Bird's surprising observation that
\begin{equation}\label{e35}
|\{F\subset [n]\,:\, F\mbox{ is Schreier and }n\in F\}|\ =\ F_n,\mbox{ for all }n\in \mathbb{N}.
\end{equation}
A relatively short proof of \eqref{e35} is to derive a formula for the left-hand side and then show that the same formula yields $F_n$. We need the well-known identity that holds for all $n\ge 1$,
\begin{equation}\label{oot}F_n\ =\ \sum_{k=0}^{\lfloor \frac{n-1}{2}\rfloor}\binom{n-1-k}{k}.\end{equation}
For a proof of \eqref{oot} using tilings, see \cite[Identity 4]{BQ}, or we may use the Zeckendorf's Theorem\footnote{Zeckendorf's Theorem states that every positive integer can be written uniquely as a sum of nonadjacent Fibonacci numbers $(F_n)_{n\ge 2}$.} to obtain another proof (see \cite[(2.5]{KKMW}).

\begin{proof}[A formula-based proof of \eqref{e35}]
Let $\mathcal{S}_n = \{F\subset [n]\,:\, F\mbox{ is Schreier and }n\in F\}$.
For each $k\ge 1$, the $k$-element sets in $\mathcal{S}_n$ are formed by choosing $k-1$ elements from $\{k, k+1, \ldots, n-1\}$. Therefore, the number of $k$-element sets is
$$\binom{(n-1)-k+1}{k-1}\ =\ \binom{n-k}{k-1},$$
under the condition $n-k\ge k-1$, i.e., $k\le (n+1)/2$. Hence, we obtain the formula
$$|\mathcal{S}_n| \ =\ \sum_{k=1}^{\left\lfloor \frac{n-1}{2}\right\rfloor} \binom{n-k}{k-1}\ =\ \sum_{k=0}^{\left\lfloor \frac{n-1}{2}\right\rfloor} \binom{n-1-k}{k},$$
which is $F_n$ according to \eqref{oot}.
\end{proof}

For more formula-based proofs, see the proofs of \cite[Theorem 9]{BC}, \cite[Theorem 1.2]{BCF}, \cite[Theorems 1 and 4]{C1}, \cite[Theorems 1--5]{CIMSZ}, and \cite[Theorems 1.1]{CV0}. While the formula-based proof of \eqref{e35} is concise, it relies on the known identity \eqref{oot} and does not explain why $(|\mathcal{S}_n|)_{n=1}^\infty$ follows the Fibonacci recurrence. This motivates the bijective proof by Bird \cite{Bi}.

\begin{proof}[A bijective proof of \eqref{e35}]
    Let $\mathcal{S}_n = \{F\subset [n]\,:\, F\mbox{ is Schreier and }n\in F\}$.
    Since $\mathcal{S}_1 = \{\{1\}\}$ and $\mathcal{S}_2 = \{\{2\}\}$, it suffices to show that $|\mathcal{S}_{n-2}| + |\mathcal{S}_{n-1}| = |\mathcal{S}_n|$ for $n\ge 3$. For $n\ge 3$, we partition $\mathcal{S}_n$ into 
    $$
        \mathcal{S}_{n,1}\ =\ \{F\in \mathcal{S}_n\,:\, n-1\in F\}\quad\mbox{ and }\quad\mathcal{S}_{n,2}\ =\ \{F\in \mathcal{S}_n\,:\, n-1\notin F\}.
    $$
    Define the two maps
    \begin{align*}
    &f: \mathcal{S}_{n-1}\rightarrow \mathcal{S}_{n, 2}, \quad\quad\quad F\mapsto (F\backslash \{n-1\})\cup \{n\}, \mbox{ and}\\
    &g: \mathcal{S}_{n-2}\rightarrow \mathcal{S}_{n,1}, \quad\quad\quad F\mapsto  (F+1)\cup \{n\}.
    \end{align*}
    
    First, both $f$ and $g$ are well-defined because 
    $$\min f(F) \ \ge\ \min F\ \ge\ |F|\ =\ |f(F)|, n\in f(F), \mbox{ and }n-1\notin f(F)$$
    and
    $$\min g(F)\ =\ \min F + 1\ \ge\ |F| + 1\ = \ |g(F)|\mbox{ and }n-1, n\in g(F).$$

    Second, it is clear from the definition that $f$ and $g$ are injective. To see that $f$ is surjective, pick $E\in \mathcal{S}_{n,2}$ and consider $F = E\backslash \{n\}\cup \{n-1\}$. Then
    $$\min F\ =\ \begin{cases} n-1,&\mbox{ if }|E| = 1\\ \min E,&\mbox{ if }|E|\ge 2\end{cases}\ \ge\ |E|\ =\ |F|,$$
    so $F\in \mathcal{S}_{n-1}$, and $f(F) = E$. 

    To see that $g$ is surjective, pick $E'\in \mathcal{S}_{n,1}$ and consider $F' = (E'\backslash \{n\})-1$. Since $n-1\in E'$, we have $\max F' = n-2$, and
    $$\min F' \ =\ \min E' - 1\ \ge\ |E'|-1\ =\ |F'|,$$
    so $F'\in \mathcal{S}_{n-2}$, and $g(F') = E'$.

    We have shown that $f$ and $g$ are bijections; hence,
    $$|\mathcal{S}_{n-2}| + |\mathcal{S}_{n-1}| \ =\ |\mathcal{S}_{n,1}| + |\mathcal{S}_{n,2}|\ =\ |\mathcal{S}_n|.$$
\end{proof}

For more bijective proofs, see \cite[Proposition 2 and Theorem 6]{BC}, \cite[Theorem 1.1]{BCF}, \cite[Theorem 3]{C1}, \cite[Theorem 1.3]{C3}, \cite[Theorem 1]{CGKMTV}, \cite[Theorems 1.1 and 1.3]{CMX}, and \cite[Theorem 1.1]{CV0}. While the above bijective proof is longer and requires some ingenuity to discover the appropriate maps, its payoff is a genuine understanding of why the sequence $(|\mathcal{S}_n|)_{n=1}^\infty$ satisfies the Fibonacci recurrence.

%%%%%%%%%%%%%%%%%%%%%%%%%%%%%%%%%%%%%%%%%%%%%%%%%%%%%%%%%%%%%%%%%%%%%%%%%%%%%%%%
%%%%%%%%%%%%%%%%%%%%%%%%%%%%%%%%%%%%%%%%%%%%%%%%%%%%%%%%%%%%%%%%%%%%%%%%%%%%%%%%%
%%%%%%%%%%%%%%%%%%%%%%%%%%%%%%%%%%%%%%%%%%%%%%%%%%%%%%%%%%%%%%%%%%%%%%%%%%%%%%%%%
%%%%%%%%%%%%%%%%%%%%%%%%%%%%%%%%%%%%%%%%%%%%%%%%%%%%%%%%%%%%%%%%%%%%%%%%%%%%%%%%%

\section{Yet more on formula-based and bijective proofs}\label{casestudy-symmetric-maximal}

As a case study in formula-based and bijective proofs, this section investigates symmetric maximal Schreier sets. We  see that the appropriate bijections are sometimes obscured until the sets are decomposed into the right pieces. Rather than presenting these maps as if they were pulled out of thin air, we show how the data and intuition naturally lead to their discovery.
\subsection{A formula-based proof}
 Recall that
$$\mathcal{M}_n\ :=\ \{F\subset [n]\,:\, F\mbox{ is symmetric maximal Schreier and }n\in F\}.$$

\begin{proof}[A formula-based proof of Theorem \ref{m1}]
    We find a formula for $|\mathcal{M}_{2n+1}(k)|$ with $k\ge 1$.

\bigskip 

\noindent Case 1:  $k$ is odd. Then $F$ takes the form $\{k, \ldots, (2n+1+k)/2,\ldots, 2n+1\}$. Since $F$ is symmetric with $|F|= k$, we construct $F$ by choosing $(k - 3)/2$ integers in $[k+1, (2n-1+k)/2]$. The number of such sets $F$ is thus 
$$\binom{(2n-1+k)/2-(k+1)+1}{(k-3)/2}\ =\ \binom{(2n-1-k)/2}{(k-3)/2}.$$

\bigskip

\noindent Case 2: $k$ is even. Then to construct $F$, we choose $(k-2)/2$ integers in $[k+1, (2n+k)/2]$. There are 
$$\binom{(2n+k)/2-(k+1)+1}{(k-2)/2}\ =\ \binom{(2n-k)/2}{(k-2)/2}$$
such sets $F$.

\bigskip 

In both cases, we may write
$$|\mathcal{M}_{2n+1}(k)|\ =\ \binom{\lfloor (2n-k)/2\rfloor}{\lfloor (k-2)/2\rfloor}\ =\ \binom{n-\lceil k/2\rceil}{\lfloor k/2\rfloor-1},$$
with
$$n-\lceil k/2\rceil\ \ge\ \lfloor k/2\rfloor - 1; \mbox{ that is}, k\le n+1.$$
Hence,
$$
|\mathcal{M}_{2n+1}|\ =\ \sum_{k=1}^{n+1}\binom{n-\lceil k/2\rceil}{\lfloor k/2\rfloor-1},
$$
and thus, it suffices to prove 
$$\sum_{k=1}^{n+1}\binom{n-\lceil k/2\rceil}{\lfloor k/2\rfloor-1}\ =\ F_n, \mbox{ for all }n\ge 0.$$
We prove the identity for $n = 2s$ with $s\ge 0$. The proof for odd $n$ is similar. We have
    \begin{align*}
        &\sum_{k=1}^{n+1} \binom{n-\lceil k/2\rceil}{\lfloor k/2 \rfloor - 1}\\
        &\ =\ \sum_{k=1}^{2s+1}\binom{2s-\lceil k/2\rceil}{\lfloor k/2\rfloor - 1}\\
        &\ =\ \sum_{j=0}^{s}\left(\binom{2s-\lceil 2j/2\rceil}{\lfloor 2j/2\rfloor - 1} + \binom{2s-\lceil (2j+1)/2\rceil}{\lfloor (2j+1)/2\rfloor - 1}\right)\\
        &\ =\ \sum_{j=0}^{s}\binom{2s-j}{j - 1} + \sum_{j=0}^s\binom{2s-j-1}{j- 1}\\
        &\ =\ \sum_{j=-1}^{s-1}\binom{2s-j-1}{j } + \sum_{j=0}^s\binom{2s-j-1}{j- 1}\\
        &\ =\ \sum_{j=1}^{s}\left(\binom{2s-j-1}{j } + \binom{2s-j-1}{j- 1}\right)+\binom{2s-1}{0} + \binom{2s-1}{-1} -\binom{s-1}{s}\\
        &\ =\ \sum_{j=1}^s\binom{2s-j}{j}  + 1 + \binom{2s-1}{-1} -\binom{s-1}{s}\\
        &\ =\ \sum_{j=0}^s\binom{2s-j}{j} + \binom{2s-1}{-1} -\binom{s-1}{s}\\
        &\ =\ F_{2s+1} + \binom{2s-1}{-1} -\binom{s-1}{s}\quad \mbox{ due to }\eqref{oot}.
        \end{align*}
        Since $\binom{2s-1}{-1} = \binom{s-1}{s} = \begin{cases}1,\mbox{  if }s = 0,\\ 0,\mbox{ if }s > 0, \end{cases}$ we obtain
        $$\sum_{k=1}^{n+1} \binom{n-\lceil k/2\rceil}{\lfloor k/2 \rfloor - 1}\ =\ F_{2s+1}\ =\ F_{n+1}.$$
\end{proof}

\subsection{A bijective proof}
For a bijective proof of Theorem \ref{m1}, we verify that
\begin{align*}
\mathcal{M}_{1} \ =\ \{\{1\}\},&\quad \mathcal{M}_3\ =\ \{\{2, 3\}\}, \quad \mathcal{M}_5\ =\ \{\{2, 5\}, \{3, 4, 5\}\},\mbox{ and}\\
&\mathcal{M}_7\ =\ \{\{2, 7\}, \{3, 5, 7\}, \{4, 5, 6, 7\}\},
\end{align*}
so it suffices to prove
\begin{equation}\label{e4}|\mathcal{M}_{2n+1}| + |\mathcal{M}_{2n+3}|\ =\ |\mathcal{M}_{2n+5}|, \mbox{ for all }n\ge 2.\end{equation}
The goal is to devise bijective maps from an appropriate partition of $\mathcal{M}_{2n+1}\cup \mathcal{M}_{2n+3}$ to a partition of $\mathcal{M}_{2n+5}$. 

\subsubsection{Data analysis and educated guesses}\label{data_analysis}
We illustrate how the data naturally lead us to the right bijective maps.
A natural map $\Psi_n$ that turns a set $F$ in $\mathcal{M}_{2n+1}$ into a set in $\mathcal{M}_{2n+5}$ should replace $2n+1$ by $2n+5$; that is, $\Psi_n$ increases the maximum of $F$ by $4$. To maintain symmetry, we keep the minimum of $F$, while increasing the integers in the middle by $2$. Tables \ref{b1} and \ref{b2} show us the sets in $\mathcal{M}_{15}, \mathcal{M}_{17}$, and $\mathcal{M}_{19}$, while the latter also gives the range of $\Psi_7:  \mathcal{M}_{15}\rightarrow\mathcal{M}_{19}$.

\begin{table}[H]
\centering

\begin{minipage}[t]{0.49\textwidth}
\centering
\scriptsize
\begin{tabular}{|c |c|l|}
\hline
Size & Count & Sets \\
\hline

2 & 1 & $\{2,15\}$ \\
\hline
3 & 1 & $\{3,9,15\}$ \\
\hline
\multirow{5}{*}{4}
& \multirow{5}{*}{5}
& $\{4,5,14,15\}$  \\
& & $\{4,6,13,15\}$ \\
& & $\{4,7,12,15\}$ \\
& & $\{4,8,11,15\}$ \\
& & $\{4,9,10,15\}$ \\
\hline
\multirow{4}{*}{5}
& \multirow{4}{*}{4}
& $\{5,6,10,14,15\}$ \\
& & $\{5,7,10,13,15\}$ \\
& & $\{5,8,10,12,15\}$ \\
& & $\{5,9,10,11,15\}$ \\
\hline
\multirow{6}{*}{6}
& \multirow{6}{*}{6}
& $\{6,7,8,13,14,15\}$ \\
& & $\{6,7,9,12,14,15\}$ \\
& & $\{6,7,10,11,14,15\}$ \\
& & $\{6,8,9,12,13,15\}$ \\
& & $\{6,8,10,11,13,15\}$ \\
& & $\{6,9,10,11,12,15\}$ \\
\hline
\multirow{3}{*}{7}
& \multirow{3}{*}{3}
& $\{7,8,9,11,13,14,15\}$ \\
& & $\{7,8,10,11,12,14,15\}$ \\
& & $\{7,9,10,11,12,13,15\}$ \\
\hline
8 & 1 & $\{8,9,10,11,12,13,14,15\}$ \\

\hline
\end{tabular}

\end{minipage}
\hfill
\begin{minipage}[t]{0.49\textwidth}
\centering
\scriptsize
\begin{tabular}{|c |c| l|}
\hline
Size & Count & Sets \\
\hline

2 & 1 & $\{2,17\}$ \\
\hline
3 & 1 & $\{3,10,17\}$ \\
\hline
\multirow{6}{*}{4}
& \multirow{6}{*}{6}
& $\{4,5,16,17\}$ \\
& & $\{4,6,15,17\}$ \\
& & $\{4,7,14,17\}$ \\
& & $\{4,8,13,17\}$ \\
& & $\{4,9,12,17\}$ \\
& & $\{4,10,11,17\}$ \\
\hline
\multirow{5}{*}{5}
& \multirow{5}{*}{5}
& $\{5,6,11,16,17\}$ \\
& & $\{5,7,11,15,17\}$ \\
& & $\{5,8,11,14,17\}$ \\
& & $\{5,9,11,13,17\}$ \\
& & $\{5,10,11,12,17\}$ \\
\hline
\multirow{10}{*}{6}
& \multirow{10}{*}{10}
& $\{6,7,8,15,16,17\}$ \\
& & $\{6,7,9,14,16,17\}$ \\
& & $\{6,7,10,13,16,17\}$ \\
& & $\{6,7,11,12,16,17\}$ \\
& & $\{6,8,9,14,15,17\}$ \\
& & $\{6,8,10,13,15,17\}$ \\
& & $\{6,8,11,12,15,17\}$ \\
& & $\{6,9,10,13,14,17\}$ \\
& & $\{6,9,11,12,14,17\}$ \\
& & $\{6,10,11,12,13,17\}$ \\
\hline
\multirow{6}{*}{7}
& \multirow{6}{*}{6}
& $\{7,8,9,12,15,16,17\}$ \\
& & $\{7,8,10,12,14,16,17\}$ \\
& & $\{7,8,11,12,13,16,17\}$ \\
& & $\{7,9,10,12,14,15,17\}$ \\
& & $\{7,9,11,12,13,15,17\}$ \\
& & $\{7,10,11,12,13,14,17\}$ \\
\hline
\multirow{4}{*}{8}
& \multirow{4}{*}{4}
& $\{8,9,10,11,14,15,16,17\}$ \\
& & $\{8,9,10,12,13,15,16,17\}$ \\
& & $\{8,9,11,12,13,14,16,17\}$ \\
& & $\{8,10,11,12,13,14,15,17\}$ \\
\hline
9 & 1 & $\{9,10,11,12,13,14,15,16,17\}$ \\

\hline
\end{tabular}

\end{minipage}

\caption{The collections $\mathcal{M}_{15}$ and $\mathcal{M}_{17}$.}
\label{b1}
\end{table}

\begin{table}[H]
\centering
\scriptsize
\begin{tabular}{|c |c |l|c|}
\hline
Size & Count & Sets & In $\Psi_7(\mathcal{M}_{15})$?\\
\hline

2 & 1 & $\{2,19\}$  & $\checkmark$\\
\hline
3 & 1 & $\{3,11,19\}$ & $\checkmark$\\
\hline
\multirow{7}{*}{4}
& \multirow{7}{*}{7}
& $\{4,5,18,19\}$ & \\
& & $\{4,6,17,19\}$ &\\
& & $\{4,7,16,19\}$ &  $\checkmark$\\
& & $\{4,8,15,19\}$ & $\checkmark$\\
& & $\{4,9,14,19\}$ & $\checkmark$\\
& & $\{4,10,13,19\}$& $\checkmark$\\
& & $\{4,11,12,19\}$ & $\checkmark$\\
\hline
\multirow{6}{*}{5}
& \multirow{6}{*}{6}
& $\{5,6,12,18,19\}$ &\\
& & $\{5,7,12,17,19\}$ & \\
& & $\{5,8,12,16,19\}$ & $\checkmark$ \\
& & $\{5,9,12,15,19\}$ & $\checkmark$\\
& & $\{5,10,12,14,19\}$ & $\checkmark$\\
& & $\{5,11,12,13,19\}$ & $\checkmark$ \\
\hline
\multirow{15}{*}{6}

& \multirow{15}{*}{15}
& $\{6,7,8,17,18,19\}$ &\\
& & $\{6,7,9,16,18,19\}$ & \\
& & $\{6,7,10,15,18,19\}$ &\\
& & $\{6,7,11,14,18,19\}$ &\\
& & $\{6,7,12,13,18,19\}$ &\\
& & $\{6,8,9,16,17,19\}$ &\\
& & $\{6,8,10,15,17,19\}$ &\\
& & $\{6,8,11,14,17,19\}$ &\\
& & $\{6,8,12,13,17,19\}$ &\\
& & $\{6,9,10,15,16,19\}$ & $\checkmark$ \\
& & $\{6,9,11,14,16,19\}$ & $\checkmark$\\
& & $\{6,9,12,13,16,19\}$ & $\checkmark$\\
& & $\{6,10,11,14,15,19\}$ & $\checkmark$\\
& & $\{6,10,12,13,15,19\}$ & $\checkmark$\\
& & $\{6,11,12,13,14,19\}$ & $\checkmark$\\
\hline
\multirow{10}{*}{7}
& \multirow{10}{*}{10}
& $\{7,8,9,13,17,18,19\}$ &\\
& & $\{7,8,10,13,16,18,19\}$ &\\
& & $\{7,8,11,13,15,18,19\}$ &\\
& & $\{7,8,12,13,14,18,19\}$ &\\
& & $\{7,9,10,13,16,17,19\}$ &\\
& & $\{7,9,11,13,15,17,19\}$ &\\
& & $\{7,9,12,13,14,17,19\}$ &\\
& & $\{7,10,11,13,15,16,19\}$ & $\checkmark$\\
& & $\{7,10,12,13,14,16,19\}$ & $\checkmark$\\
& & $\{7,11,12,13,14,15,19\}$ & $\checkmark$\\
\hline
\multirow{10}{*}{8}
& \multirow{10}{*}{10}
& $\{8,9,10,11,16,17,18,19\}$ &\\
& & $\{8,9,10,12,15,17,18,19\}$ &\\
& & $\{8,9,10,13,14,17,18,19\}$ &\\
& & $\{8,9,11,12,15,16,18,19\}$ &\\
& & $\{8,9,11,13,14,16,18,19\}$ &\\
& & $\{8,9,12,13,14,15,18,19\}$ &\\
& & $\{8,10,11,12,15,16,17,19\}$ &\\
& & $\{8,10,11,13,14,16,17,19\}$ &\\
& & $\{8,10,12,13,14,15,17,19\}$ &\\
& & $\{8,11,12,13,14,15,16,19\}$ & $\checkmark$\\
\hline
\multirow{4}{*}{9}
& \multirow{4}{*}{4}
& $\{9,10,11,12,14,16,17,18,19\}$ & \\
& & $\{9,10,11,13,14,15,17,18,19\}$ &\\
& & $\{9,10,12,13,14,15,16,18,19\}$ &\\
& & $\{9,11,12,13,14,15,16,17,19\}$ &\\
\hline
10 & 1 & $\{10,11,12,13,14,15,16,17,18,19\}$ &\\

\hline
\end{tabular}
\caption{The collection $\mathcal{M}_{19}$.}
\label{b2}
\end{table}

We see that a set in $\mathcal{M}_{19}$ is in the range of $\Psi_7$ if and only if  the smallest and the second smallest integers in the set are at least $3$ apart. This motivates us to partition $\mathcal{M}_{19}$ and in general, $\mathcal{M}_{2n+5}$ into two subcollections based on the property. 

Given a nonempty, finite set $A\subset \mathbb{N}$ and $k\in \mathbb{N}$, let $\min_k A$ and $\max_k A$ be the $k$th smallest element and the $k$th largest element of $A$, respectively. 
We partition $\mathcal{M}_{2n+5}$ into 
$$\mathcal{M}^{(1)}_{2n+5}\ :=\ \left\{F\subset [2n+5]\,:\, 
\begin{matrix} \min_2 F - \min F \le 2, 2n+5\in F,\mbox{ and}\\
          F\mbox{ is maximal Schreier and symmetric}
\end{matrix}\right\}$$
and 
$$\mathcal{M}^{(2)}_{2n+5}\ :=\ \left\{F\subset [2n+5]\,:\, 
\begin{matrix} \min_2 F - \min F \ge 3, 2n+5\in F,\mbox{ and}\\
          F\mbox{ is maximal Schreier and symmetric}
\end{matrix}\right\}.$$

The above analysis guides us towards showing
\begin{equation}\label{e51}
    |\mathcal{M}_{2n+1}| \ =\ |\mathcal{M}^{(2)}_{2n+5}|\\
\end{equation}
and
\begin{equation}\label{e52}
    |\mathcal{M}_{2n+3}| \ =\ |\mathcal{M}^{(1)}_{2n+5}|,
\end{equation}
where \eqref{e51} can be seen through the map $\Psi_n: \mathcal{M}_{2n+1}\rightarrow\mathcal{M}^{(2)}_{2n+5}$ defined as
$$F\ \mapsto\ \{\min F\}\cup \left(\{a\in F: \min F < a < 2n+1\}+2\right)\cup \{2n+5\}.$$

We now investigate how the sets in $\mathcal{M}_{2n+3}$ can be mapped to the sets in $\mathcal{M}^{(1)}_{2n+5}$ with an eye on $n = 7$. Table \ref{b3} updates the data accordingly.

\begin{table}[H]
\centering

\begin{minipage}[t]{0.49\textwidth}
\centering
\scriptsize
\begin{tabular}{|c |c| l|}
\hline
Size & Count & Sets \\
\hline

2 & 1 & $\{2,17\}$ \\
\hline
3 & 1 & $\{3,10,17\}$ \\
\hline
\multirow{6}{*}{4}
& \multirow{6}{*}{6}
& $\{4,5,16,17\}$ \\
& & $\{4,6,15,17\}$ \\
& & $\{4,7,14,17\}$ \\
& & $\{4,8,13,17\}$ \\
& & $\{4,9,12,17\}$ \\
& & $\{4,10,11,17\}$ \\
\hline
\multirow{5}{*}{5}
& \multirow{5}{*}{5}
& $\{5,6,11,16,17\}$ \\
& & $\{5,7,11,15,17\}$ \\
& & $\{5,8,11,14,17\}$ \\
& & $\{5,9,11,13,17\}$ \\
& & $\{5,10,11,12,17\}$ \\
\hline
\multirow{10}{*}{6}
& \multirow{10}{*}{10}
& $\{6,7,8,15,16,17\}$ \\
& & $\{6,7,9,14,16,17\}$ \\
& & $\{6,7,10,13,16,17\}$ \\
& & $\{6,7,11,12,16,17\}$ \\
& & $\{6,8,9,14,15,17\}$ \\
& & $\{6,8,10,13,15,17\}$ \\
& & $\{6,8,11,12,15,17\}$ \\
& & $\{6,9,10,13,14,17\}$ \\
& & $\{6,9,11,12,14,17\}$ \\
& & $\{6,10,11,12,13,17\}$ \\
\hline
\multirow{6}{*}{7}
& \multirow{6}{*}{6}
& $\{7,8,9,12,15,16,17\}$ \\
& & $\{7,8,10,12,14,16,17\}$ \\
& & $\{7,8,11,12,13,16,17\}$ \\
& & $\{7,9,10,12,14,15,17\}$ \\
& & $\{7,9,11,12,13,15,17\}$ \\
& & $\{7,10,11,12,13,14,17\}$ \\
\hline
\multirow{4}{*}{8}
& \multirow{4}{*}{4}
& $\{8,9,10,11,14,15,16,17\}$ \\
& & $\{8,9,10,12,13,15,16,17\}$ \\
& & $\{8,9,11,12,13,14,16,17\}$ \\
& & $\{8,10,11,12,13,14,15,17\}$ \\
\hline
9 & 1 & $\{9,10,11,12,13,14,15,16,17\}$ \\

\hline
\end{tabular}
\end{minipage}
\hfill
\begin{minipage}[t]{0.49\textwidth}
\centering
\scriptsize
\begin{tabular}{|c |c| l|}
\hline
Size & Count & Sets \\
\hline

\hline
\multirow{2}{*}{4}
& \multirow{2}{*}{2}
& $\{4,5,18,19\}$  \\
& & $\{4,6,17,19\}$ \\
\hline
\multirow{2}{*}{5}
& \multirow{2}{*}{2}
& $\{5,6,12,18,19\}$ \\
& & $\{5,7,12,17,19\}$ \\
\hline
\multirow{9}{*}{6}
& \multirow{9}{*}{9}
& $\{6,7,8,17,18,19\}$ \\
& & $\{6,7,9,16,18,19\}$  \\
& & $\{6,7,10,15,18,19\}$ \\
& & $\{6,7,11,14,18,19\}$ \\
& & $\{6,7,12,13,18,19\}$ \\
& & $\{6,8,9,16,17,19\}$ \\
& & $\{6,8,10,15,17,19\}$ \\
& & $\{6,8,11,14,17,19\}$ \\
& & $\{6,8,12,13,17,19\}$ \\
\hline
\multirow{7}{*}{7}
& \multirow{7}{*}{7}
& $\{7,8,9,13,17,18,19\}$ \\
& & $\{7,8,10,13,16,18,19\}$ \\
& & $\{7,8,11,13,15,18,19\}$ \\
& & $\{7,8,12,13,14,18,19\}$ \\
& & $\{7,9,10,13,16,17,19\}$ \\
& & $\{7,9,11,13,15,17,19\}$ \\
& & $\{7,9,12,13,14,17,19\}$ \\
\hline
\multirow{9}{*}{8}
& \multirow{9}{*}{9}
& $\{8,9,10,11,16,17,18,19\}$ \\
& & $\{8,9,10,12,15,17,18,19\}$ \\
& & $\{8,9,10,13,14,17,18,19\}$ \\
& & $\{8,9,11,12,15,16,18,19\}$ \\
& & $\{8,9,11,13,14,16,18,19\}$ \\
& & $\{8,9,12,13,14,15,18,19\}$ \\
& & $\{8,10,11,12,15,16,17,19\}$ \\
& & $\{8,10,11,13,14,16,17,19\}$ \\
& & $\{8,10,12,13,14,15,17,19\}$ \\
\hline
\multirow{4}{*}{9}
& \multirow{4}{*}{4}
& $\{9,10,11,12,14,16,17,18,19\}$  \\
& & $\{9,10,11,13,14,15,17,18,19\}$ \\
& & $\{9,10,12,13,14,15,16,18,19\}$ \\
& & $\{9,11,12,13,14,15,16,17,19\}$ \\
\hline
10 & 1 & $\{10,11,12,13,14,15,16,17,18,19\}$ \\
\hline
\end{tabular}
\end{minipage}

\caption{The collections $\mathcal{M}_{17}$ and $\mathcal{M}^{(1)}_{19}$}
\label{b3}
\end{table}

While $|\mathcal{M}_{17}| = |\mathcal{M}^{(1)}_{19}| = 34$, sets of the same size in the two collections may come in different numbers.  In particular,   
\begin{align*}
    |\mathcal{M}_{17}|&\ =\ 1 + 1 + 6 + 5 + 10 + 6 + 4 + 1\ =\ 34,\mbox{ and }\\
    |\mathcal{M}^{(1)}_{19}|&\ =\ 2+ 2+ 9 + 7 + 9 + 4+ 1\ =\ 34.
\end{align*}
This phenomenon suggests that a deeper structure may be at work. We should examine these numbers more closely and uncover the source of the equality. The two decompositions of $34$ hint at the grouping:

\begin{table}[H]
\centering
\begin{tabular}{ |c| c| c| c| c| c||| c| c| c||| c| c| c| c| c| c|c|c| }
\hline
 $k$ &$2$& $3$ & $4$ & $5$ & $6$ & $6$ & $7$ & $8$ & $8$ & $9$ & $10$  \\
\hline
$|\mathcal{M}_{17}(k)|$ & $1$ & $1$ & $6$ & $5$ &  & $10$ & $6$ &  & $4$ & $1$ & \\
\hline
$|\mathcal{M}^{(1)}_{19}(k)|$ & $$ & $$ & $2$ & $2$ & $9$ &  & $7$ & $9$ &  & $4$ & $1$  \\
\hline
\end{tabular}
\caption{A possibly helpful grouping of $\mathcal{M}_{17}$ and $\mathcal{M}^{(1)}_{19}$.}
\end{table}
We suspect that for $n\ge 2$ and even $k\ge 6$, 
\begin{equation}\label{e45}
    |\mathcal{M}_{2n+3}(k)| + |\mathcal{M}_{2n+3}(k+1)| \ =\ |\mathcal{M}^{(1)}_{2n+5}(k+1)| + |\mathcal{M}^{(1)}_{2n+5}(k+2)|.
\end{equation}
Our suspicion is strongly supported by the data for different collections $\mathcal{M}_{2n+3}$ and $\mathcal{M}^{(1)}_{2n+5}$.

\begin{table}[H]
\centering
\begin{tabular}{ |c| c| c| c| c| c| c| c| c| c| c| c| c| c| c|c| }
\hline
 $k$ &$2$& $3$ & $4$ & $5$ & $6$ & $7$ & $8$ & $9$ & $10$ & $11$ & $12$ & $13$ & $14$ & $15$  \\
\hline
$|\mathcal{M}_{7}(k)|$ & $1$ & $1$ & $1$ &  &  &  &  &  &  &  & & & & \\
\hline
$|\mathcal{M}_{9}(k)|$ & $1$ & $1$ & $2$ & $1$ &  &  &  &  &  &  & & & &   \\
\hline
$|\mathcal{M}_{11}(k)|$ & $1$ & $1$ & $3$ & $2$ & $1$ &  &  &  &  &  & & & &   \\
\hline
$|\mathcal{M}_{13}(k)|$ & $1$ & $1$ & $4$ & $3$ & $3$ & $1$ &  &  &  &  & & & &   \\
\hline
$|\mathcal{M}_{15}(k)|$ & $1$ & $1$ & $5$ & $4$ & $6$ & $3$ & $1$ &  &  &  & & & &   \\
\hline 
$|\mathcal{M}_{17}(k)|$ & $1$ & $1$ & $6$ & $5$ & $10$ & $6$ & $4$ & $1$ &  &  & & & &   \\
\hline 
$|\mathcal{M}_{19}(k)|$ & $1$ & $1$ & $7$ & $6$ & $15$ & $10$ & $10$ & $4$ & $1$ &  & & & &   \\
\hline 
$|\mathcal{M}_{21}(k)|$ & $1$ & $1$ & $8$ & $7$ & $21$ & $15$ & $20$ & $10$ & $5$ & $1$  & & & &   \\
\hline 
$|\mathcal{M}_{23}(k)|$ & $1$ & $1$ & $9$ & $8$ & $28$ & $21$ & $35$ & $20$ & $15$ & $5$  & $1$ & & &   \\
\hline 
$|\mathcal{M}_{25}(k)|$ & $1$ & $1$ & $10$ & $9$ & $36$ & $28$ & $56$ & $35$ & $35$ & $15$  & $6$ & $1$ & &   \\
\hline 
$|\mathcal{M}_{27}(k)|$ & $1$ & $1$ & $11$ & $10$ & $45$ & $36$ & $84$ & $56$ & $70$ & $35$  & $21$ & $6$ & $1$ &   \\
\hline 
$|\mathcal{M}_{29}(k)|$ & $1$ & $1$ & $12$ & $11$ & $55$ & $45$ & $120$ & $84$ & $126$ & $70$  & $56$ & $21$ & $7$ & $1$   \\
\hline 
\end{tabular}
\caption{Breakdown of $\mathcal{M}_{2n+3}$ based on set sizes.}
\label{table10}
\end{table}

\begin{table}[H]
\centering
\begin{tabular}{ |c| c| c| c| c| c| c| c| c| c| c| c| c| c| c| }
\hline
 $k$ &$4$& $5$ & $6$ & $7$ & $8$ & $9$ & $10$ & $11$ & $12$ & $13$ & $14$ & $15$ & $16$   \\
\hline
$|\mathcal{M}^{(1)}_{9}(k)|$ & $2$ & $1$ &  &  &  &  &  &  &  &  & & &  \\
\hline
$|\mathcal{M}^{(1)}_{11}(k)|$ & $2$ & $2$ & $1$ &  &  &  &  &  &  &  & & &    \\
\hline
$|\mathcal{M}^{(1)}_{13}(k)|$ & $2$ & $2$ & $3$ & $1$ &  &  &  &  &  &  & & &    \\
\hline
$|\mathcal{M}^{(1)}_{15}(k)|$ & $2$ & $2$ & $5$ & $3$ & $1$ & &  &  &  &  & & &   \\
\hline
$|\mathcal{M}^{(1)}_{17}(k)|$ & $2$ & $2$ & $7$ & $5$ & $4$ & $1$ &  &  &  &  & & &    \\
\hline 
$|\mathcal{M}^{(1)}_{19}(k)|$ & $2$ & $2$ & $9$ & $7$ & $9$ & $4$ & $1$ &  &  &  & & &    \\
\hline 
$|\mathcal{M}^{(1)}_{21}(k)|$ & $2$ & $2$ & $11$ & $9$ & $16$ & $9$ & $5$ & $1$ &  &  & & &   \\
\hline 
$|\mathcal{M}^{(1)}_{23}(k)|$ & $2$ & $2$ & $13$ & $11$ & $25$ & $16$ & $14$ & $5$ & $1$ &   & & &    \\
\hline 
$|\mathcal{M}^{(1)}_{25}(k)|$ & $2$ & $2$ & $15$ & $13$ & $36$ & $25$ & $30$ & $14$ & $6$ & $1$  &  & &   \\
\hline 
$|\mathcal{M}^{(1)}_{27}(k)|$ & $2$ & $2$ & $17$ & $15$ & $49$ & $36$ & $55$ & $30$ & $20$ & $6$  & $1$ &  &   \\
\hline 
$|\mathcal{M}^{(1)}_{29}(k)|$ & $2$ & $2$ & $19$ & $17$ & $64$ & $49$ & $91$ & $55$ & $50$ & $20$  & $7$ & $1$ &     \\
\hline 
$|\mathcal{M}^{(1)}_{31}(k)|$ & $2$ & $2$ & $21$ & $19$ & $81$ & $64$ & $140$ & $91$ & $105$ & $50$  & $27$ & $7$ & $1$     \\
\hline 
\end{tabular}
\caption{Breakdown of $\mathcal{M}^{(1)}_{2n+5}$ based on set sizes.}
\label{table11}
\end{table}

Tables \ref{table10} and \ref{table11} confirm the conjectured Identity \eqref{e45} for all $n$ up to $13$: for example, 
\begin{itemize}
    \item when $(n, k) = (8, 6)$, we have 
    $$|\mathcal{M}_{19}(6)| + |\mathcal{M}_{19}(7)| \ =\ 15 + 10\  =\ 9+ 16 \ = \ |\mathcal{M}^{(1)}_{21}(7)| + |\mathcal{M}^{(1)}_{21}(8)|;$$
    \item when $(n,k) = (11, 8)$, we have
        $$|\mathcal{M}_{25}(8)| + |\mathcal{M}_{25}(9)| \ =\ 56 + 35\  =\ 36+ 55 \ = \ |\mathcal{M}^{(1)}_{27}(9)| + |\mathcal{M}^{(1)}_{27}(10)|;$$
    \item when $(n,k) = (13, 10)$, we have
        $$|\mathcal{M}_{29}(10)| + |\mathcal{M}_{29}(11)| \ =\ 126 + 70\  =\ 91+ 105 \ = \ |\mathcal{M}^{(1)}_{31}(11)| + |\mathcal{M}^{(1)}_{31}(12)|.$$
\end{itemize}
We now construct bijective maps that show \eqref{e45}. Let us see what the data for $n = 10$ suggest by comparing the sets in $\mathcal{M}_{23}(8) \cup \mathcal{M}_{23}(9)$ with the ones in $\mathcal{M}^{(1)}_{25}(9)\cup \mathcal{M}^{(1)}_{25}(10)$.

\begin{table}[H]
\centering

\begin{minipage}[t]{0.49\textwidth}
\centering
\scriptsize
\begin{tabular}{|c|c|}
\hline
 Subcollection & $\mathcal{M}_{23}(8)$\\
\hline
 \multirow{25}{*}{$A$}& \{8, 9, 10, 11, 20, 21, 22, 23\}\\
 &\{8, 9, 10, 12, 19, 21, 22, 23\}\\
&\{8, 9, 10, 13, 18, 21, 22, 23\}\\
&\{8, 9, 10, 14, 17, 21, 22, 23\}\\
&\{8, 9, 10, 15, 16, 21, 22, 23\}\\
& \{8, 9, 11, 12, 19, 20, 22, 23\}\\
&\{8, 9, 11, 13, 18, 20, 22, 23\}\\
&\{8, 9, 11, 14, 17, 20, 22, 23\}\\
&\{8, 9, 11, 15, 16, 20, 22, 23\}\\
&\{8, 9, 12, 13, 18, 19, 22, 23\}\\
&\{8, 9, 12, 14, 17, 19, 22, 23\}\\
&\{8, 9, 12, 15, 16, 19, 22, 23\}\\
&\{8, 9, 13, 14, 17, 18, 22, 23\}\\
&\{8, 9, 13, 15, 16, 18, 22, 23\}\\
&\{8, 9, 14, 15, 16, 17, 22, 23\}\\
&\{8, 10, 11, 12, 19, 20, 21, 23\}\\
&\{8, 10, 11, 13, 18, 20, 21, 23\}\\
&\{8, 10, 11, 14, 17, 20, 21, 23\}\\
&\{8, 10, 11, 15, 16, 20, 21, 23\}\\
&\{8, 10, 12, 13, 18, 19, 21, 23\}\\
&\{8, 10, 12, 14, 17, 19, 21, 23\}\\
&\{8, 10, 12, 15, 16, 19, 21, 23\}\\
&\{8, 10, 13, 14, 17, 18, 21, 23\}\\
&\{8, 10, 13, 15, 16, 18, 21, 23\}\\
&\{8, 10, 14, 15, 16, 17, 21, 23\}\\
\hline
\multirow{10}{*}{$B$} &\{8, 11, 12, 13, 18, 19, 20, 23\}\\
&\{8, 11, 12, 14, 17, 19, 20, 23\}\\
&\{8, 11, 12, 15, 16, 19, 20, 23\}\\
&\{8, 11, 13, 14, 17, 18, 20, 23\}\\
&\{8, 11, 13, 15, 16, 18, 20, 23\}\\
&\{8, 11, 14, 15, 16, 17, 20, 23\}\\
&\{8, 12, 13, 14, 17, 18, 19, 23\}\\
&\{8, 12, 13, 15, 16, 18, 19, 23\}\\
&\{8, 12, 14, 15, 16, 17, 19, 23\}\\
&\{8, 13, 14, 15, 16, 17, 18, 23\}\\
\hline
Subcollection & $\mathcal{M}_{23}(9)$\\
\hline
\multirow{16}{*}{$C$}&\{9, 10, 11, 12, 16, 20, 21, 22, 23\}\\
&\{9, 10, 11, 13, 16, 19, 21, 22, 23\}\\
&\{9, 10, 11, 14, 16, 18, 21, 22, 23\}\\
&\{9, 10, 11, 15, 16, 17, 21, 22, 23\}\\
&\{9, 10, 12, 13, 16, 19, 20, 22, 23\}\\
&\{9, 10, 12, 14, 16, 18, 20, 22, 23\}\\
&\{9, 10, 12, 15, 16, 17, 20, 22, 23\}\\
&\{9, 10, 13, 14, 16, 18, 19, 22, 23\}\\
&\{9, 10, 13, 15, 16, 17, 19, 22, 23\}\\
&\{9, 10, 14, 15, 16, 17, 18, 22, 23\}\\
&\{9, 11, 12, 13, 16, 19, 20, 21, 23\}\\
&\{9, 11, 12, 14, 16, 18, 20, 21, 23\}\\
&\{9, 11, 12, 15, 16, 17, 20, 21, 23\}\\
&\{9, 11, 13, 14, 16, 18, 19, 21, 23\}\\
&\{9, 11, 13, 15, 16, 17, 19, 21, 23\}\\
&\{9, 11, 14, 15, 16, 17, 18, 21, 23\}\\
\hline
\multirow{4}{*}{$D$} &\{9, 12, 13, 14, 16, 18, 19, 20, 23\}\\
&\{9, 12, 13, 15, 16, 17, 19, 20, 23\}\\
&\{9, 12, 14, 15, 16, 17, 18, 20, 23\}\\
&\{9, 13, 14, 15, 16, 17, 18, 19, 23\}\\
\hline
\end{tabular}
\end{minipage}
\hfill
\begin{minipage}[t]{0.49\textwidth}
\centering
\scriptsize
\begin{tabular}{|c|c|}
\hline
Subcollection & $\mathcal{M}^{(1)}_{25}(9)$\\
\hline
\multirow{25}{*}{$W$}&\{9, 10, 11, 12, 17, 22, 23, 24, 25\}\\
&\{9, 10, 11, 13, 17, 21, 23, 24, 25\}\\
&\{9, 10, 11, 14, 17, 20, 23, 24, 25\}\\
&\{9, 10, 11, 15, 17, 19, 23, 24, 25\}\\
&\{9, 10, 11, 16, 17, 18, 23, 24, 25\}\\
&\{9, 10, 12, 13, 17, 21, 22, 24, 25\}\\
&\{9, 10, 12, 14, 17, 20, 22, 24, 25\}\\
&\{9, 10, 12, 15, 17, 19, 22, 24, 25\}\\
&\{9, 10, 12, 16, 17, 18, 22, 24, 25\}\\
&\{9, 10, 13, 14, 17, 20, 21, 24, 25\}\\
&\{9, 10, 13, 15, 17, 19, 21, 24, 25\}\\
&\{9, 10, 13, 16, 17, 18, 21, 24, 25\}\\
&\{9, 10, 14, 15, 17, 19, 20, 24, 25\}\\
&\{9, 10, 14, 16, 17, 18, 20, 24, 25\}\\
&\{9, 10, 15, 16, 17, 18, 19, 24, 25\}\\
&\{9, 11, 12, 13, 17, 21, 22, 23, 25\}\\
&\{9, 11, 12, 14, 17, 20, 22, 23, 25\}\\
&\{9, 11, 12, 15, 17, 19, 22, 23, 25\}\\
&\{9, 11, 12, 16, 17, 18, 22, 23, 25\}\\
&\{9, 11, 13, 14, 17, 20, 21, 23, 25\}\\
&\{9, 11, 13, 15, 17, 19, 21, 23, 25\}\\
&\{9, 11, 13, 16, 17, 18, 21, 23, 25\}\\
&\{9, 11, 14, 15, 17, 19, 20, 23, 25\}\\
&\{9, 11, 14, 16, 17, 18, 20, 23, 25\}\\
&\{9, 11, 15, 16, 17, 18, 19, 23, 25\}\\
\hline
Subcollection&$\mathcal{M}^{(1)}_{25}(10)$\\
\hline
$X$ & \{10, 11, 12, 13, 14, 21, 22, 23, 24, 25\}\\
 &\{10, 11, 12, 13, 15, 20, 22, 23, 24, 25\}\\
 &\{10, 11, 12, 13, 16, 19, 22, 23, 24, 25\}\\
 &\{10, 11, 12, 13, 17, 18, 22, 23, 24, 25\}\\
$X$ & \{10, 11, 12, 14, 15, 20, 21, 23, 24, 25\}\\
 &\{10, 11, 12, 14, 16, 19, 21, 23, 24, 25\}\\
 &\{10, 11, 12, 14, 17, 18, 21, 23, 24, 25\}\\
$X$ & \{10, 11, 12, 15, 16, 19, 20, 23, 24, 25\}\\
 &\{10, 11, 12, 15, 17, 18, 20, 23, 24, 25\}\\
$X$ &\{10, 11, 12, 16, 17, 18, 19, 23, 24, 25\}\\
$X$ &\{10, 11, 13, 14, 15, 20, 21, 22, 24, 25\}\\
 &\{10, 11, 13, 14, 16, 19, 21, 22, 24, 25\}\\
 &\{10, 11, 13, 14, 17, 18, 21, 22, 24, 25\}\\
$X$ & \{10, 11, 13, 15, 16, 19, 20, 22, 24, 25\}\\
&\{10, 11, 13, 15, 17, 18, 20, 22, 24, 25\}\\
$X$ & \{10, 11, 13, 16, 17, 18, 19, 22, 24, 25\}\\
$X$ & \{10, 11, 14, 15, 16, 19, 20, 21, 24, 25\}\\
&\{10, 11, 14, 15, 17, 18, 20, 21, 24, 25\}\\
$X$ & \{10, 11, 14, 16, 17, 18, 19, 21, 24, 25\}\\
$X$ & \{10, 11, 15, 16, 17, 18, 19, 20, 24, 25\}\\
$X$ & \{10, 12, 13, 14, 15, 20, 21, 22, 23, 25\}\\
 &\{10, 12, 13, 14, 16, 19, 21, 22, 23, 25\}\\
 &\{10, 12, 13, 14, 17, 18, 21, 22, 23, 25\}\\
$X$ & \{10, 12, 13, 15, 16, 19, 20, 22, 23, 25\}\\
 &\{10, 12, 13, 15, 17, 18, 20, 22, 23, 25\}\\
$X$ & \{10, 12, 13, 16, 17, 18, 19, 22, 23, 25\}\\
$X$ & \{10, 12, 14, 15, 16, 19, 20, 21, 23, 25\}\\
 & \{10, 12, 14, 15, 17, 18, 20, 21, 23, 25\}\\
$X$ & \{10, 12, 14, 16, 17, 18, 19, 21, 23, 25\}\\
$X$ & \{10, 12, 15, 16, 17, 18, 19, 20, 23, 25\}\\
\hline
\end{tabular}
\end{minipage}

\caption{Sets in $\mathcal{M}_{23}(8) \cup \mathcal{M}_{23}(9)$ and $\mathcal{M}^{(1)}_{25}(9)\cup \mathcal{M}^{(1)}_{25}(10)$.}
\label{b4}
\end{table}

We would like to map sets $F$ in $\mathcal{M}_{23}(8)$ of the form $\{8, \ldots, 23\}$ to corresponding sets $E$ in $\mathcal{M}^{(1)}_{25}(9)$ of the form $\{9, \ldots, 25\}$. The change of the minimum from $8$ to $9$ and of the maximum from $23$ to $25$, in combination with symmetry and maximality, suggests that we should add $1$ to the smaller half of $F$, add $2$ to the larger half of $F$, and then append $(9+25)/2 = 17$. This leads to the bijection from $A$ into $W$:
\begin{equation}\label{e60}F\ \mapsto\ (\{a\in F: a < s_F\}+1)\cup \{s_F + 3/2\}\cup (\{a\in F: a > s_F\}+2).\end{equation}
Similarly, we would like to map sets $F$ in $\mathcal{M}^{(1)}_{25}(10)$  to corresponding sets $E$ in $\mathcal{M}_{23}(9)$. The change of the minimum from $10$ to $9$ and of the maximum from $25$ to $23$, in combination with symmetry and the decrease in cardinality by $1$, suggests that we should subtract $1$ from the smaller half of $F$, subtract $2$ from the larger half of $F$, discard two integers in the middle, and then append $(23+9)/2 = 16$. 
This leads to the bijection from $X$ into $C$:
\begin{equation}\label{e61}F\ \mapsto\ (\{a\in F: a\le \min_{4} F\} - 1)\cup\{(s_F-3)/2\}\cup (\{a\in F: a \ge \max_{4} F\}-2).\end{equation}

Observe that a set $F\in \mathcal{M}_{23}(8)$ is in $A$ if and only if $\min_2 F - \min F \le 2$; a set $F\in \mathcal{M}_{23}(9)$ is in $C$ if and only if $\min_2 F-\min F\le 2$; finally, a set $F\in \mathcal{M}^{(1)}_{25}(10)$ is in $X$ if and only if $\min_{5} F - \min_{4}F = 1$. These motivates the definition: given a collection $\mathcal{A}$ of sets of numbers, define
$$\mathcal{A}^{(t_1, \le d_1), \ldots, (t_p, \le d_p)} \ :=\ \{F\in\mathcal{A}\,:\, \min_{t_i+1}-\min_{t_i} \le d_i,\mbox{ for all }i\le p\}.$$
The inequality $\le$ can be replaced by $=$, $\ge$, $<$, and $>$.
In particular, 
\begin{align*}
\mathcal{M}_{2n+3}(k)^{(1, \le 2)}& := \{F\in \mathcal{M}_{2n+3}(k)\,:\, \min_2 F - \min F \le 2\},\\
\mathcal{M}_{2n+3}(k+1)^{(1, \le 2)}& := \{F\in \mathcal{M}_{2n+3}(k+1)\,:\, \min_2 F - \min F \le 2\},\\
\mathcal{M}_{2n+5}(k+1)^{(1, \le 2)}& := \{F\in \mathcal{M}_{2n+5}(k+1)\,:\, \min_2 F - \min F \le 2\},\\
& \mbox{ which is the same as }\mathcal{M}^{(1)}_{2n+5}(k+1)\mbox{ used in \eqref{e45}, and}\\
\mathcal{M}_{2n+5}(k+2)^{(1, \le 2), \left(\frac{k}{2}, =1\right)} & := \left\{F\in \mathcal{M}_{2n+5}(k+2)\,:\, \begin{matrix}\min_2 F-\min F\le 2,\mbox{ and}\\ \min_{k/2+1} F - \min_{k/2} F = 1\end{matrix}\right\}.
\end{align*}
We can now write the maps \eqref{e60} and \eqref{e61} for general $n\ge 2$ and even $k\ge 6$:
\begin{align*}
&\Theta_{n,k}: \mathcal{M}_{2n+3}(k)^{(1, \le 2)}\ \rightarrow\ \mathcal{M}_{2n+5}(k+1)^{(1,\le 2)},\\
&\quad F\ \mapsto\ (\{a\in F: a < s_F\}+1)\cup \{s_F + 3/2\}\cup (\{a\in F: a > s_F\}+2)
\end{align*}
and 
\begin{align*}
&\Gamma_{n,k}:     \mathcal{M}_{2n+5}(k+2)^{(1, \le 2), \left(\frac{k}{2}, =1\right)}\ \rightarrow\ \mathcal{M}_{2n+3}(k+1)^{(1, \le 2)},\\
&\quad F\ \mapsto\ (\{a\in F: a\le \min_{k/2} F\} - 1)\cup\{(s_F-3)/2\}\cup (\{a\in F: a \ge \max_{k/2} F\}-2).
\end{align*}

Next, let us clean up Table \ref{b4} by removing the subcollections $A$, $C$, $W$, and $X$ to obtain Table \ref{b5}, which consists of 
\begin{align*}
    \mathcal{M}_{23}(8)^{(1, \ge 3)} &:= \{F\in \mathcal{M}_{23}(8)\,:\, \min_2 F - \min F \ge 3\},\\
\mathcal{M}_{23}(9)^{(1, \ge 3)} &:= \{F\in \mathcal{M}_{23}(9)\,:\, \min_2 F - \min F \ge 3\},\mbox{ and}\\
\mathcal{M}_{25}(10)^{(1, \le 2), \left(4, \ge 2\right)} & := \left\{F\in \mathcal{M}_{25}(10)\,:\, \begin{matrix}\min_2 F-\min F\le 2,\mbox{ and}\\ \min_{5} F - \min_{4} F \ge 2\end{matrix}\right\}.
\end{align*}

\begin{table}[H]
\centering

\begin{minipage}[t]{0.49\textwidth}
\centering
\scriptsize
\begin{tabular}{|c|}
\hline
$\mathcal{M}_{23}(8)^{(1, \ge 3)}$\\
\hline
\{8, 11, 12, 13, 18, 19, 20, 23\}\\
\{8, 11, 12, 14, 17, 19, 20, 23\}\\
\{8, 11, 12, 15, 16, 19, 20, 23\}\\
\{8, 11, 13, 14, 17, 18, 20, 23\}\\
\{8, 11, 13, 15, 16, 18, 20, 23\}\\
\{8, 11, 14, 15, 16, 17, 20, 23\}\\
\{8, 12, 13, 14, 17, 18, 19, 23\}\\
\{8, 12, 13, 15, 16, 18, 19, 23\}\\
\{8, 12, 14, 15, 16, 17, 19, 23\}\\
\{8, 13, 14, 15, 16, 17, 18, 23\}\\
\hline
$\mathcal{M}_{23}(9)^{(1, \ge 3)}$\\
\hline
\{9, 12, 13, 14, 16, 18, 19, 20, 23\}\\
\{9, 12, 13, 15, 16, 17, 19, 20, 23\}\\
\{9, 12, 14, 15, 16, 17, 18, 20, 23\}\\
\{9, 13, 14, 15, 16, 17, 18, 19, 23\}\\
\hline
\end{tabular}
\end{minipage}
\hfill
\begin{minipage}[t]{0.49\textwidth}
\centering
\scriptsize
\begin{tabular}{|c|}
\hline
$\mathcal{M}_{25}(10)^{(1, \le 2), (4, \ge 2)}$\\
\hline
 \{10, 11, 12, 13, 15, 20, 22, 23, 24, 25\}\\
 \{10, 11, 12, 13, 16, 19, 22, 23, 24, 25\}\\
 \{10, 11, 12, 13, 17, 18, 22, 23, 24, 25\}\\
 \{10, 11, 12, 14, 16, 19, 21, 23, 24, 25\}\\
 \{10, 11, 12, 14, 17, 18, 21, 23, 24, 25\}\\
 \{10, 11, 12, 15, 17, 18, 20, 23, 24, 25\}\\
 \{10, 11, 13, 14, 16, 19, 21, 22, 24, 25\}\\
 \{10, 11, 13, 14, 17, 18, 21, 22, 24, 25\}\\
\{10, 11, 13, 15, 17, 18, 20, 22, 24, 25\}\\
\{10, 11, 14, 15, 17, 18, 20, 21, 24, 25\}\\
 \{10, 12, 13, 14, 16, 19, 21, 22, 23, 25\}\\
 \{10, 12, 13, 14, 17, 18, 21, 22, 23, 25\}\\
 \{10, 12, 13, 15, 17, 18, 20, 22, 23, 25\}\\
  \{10, 12, 14, 15, 17, 18, 20, 21, 23, 25\}\\
\hline
\end{tabular}
\end{minipage}

\caption{Sets in the subcollections $\mathcal{M}_{23}(8)^{(1, \ge 3)}$, $\mathcal{M}_{23}(9)^{(1, \ge 3)}$, and
$\mathcal{M}_{25}(10)^{(1, \le 2), (4, \ge 2)}$.}
\label{b5}
\end{table}

Table \ref{b5} guides us towards splitting $\mathcal{M}_{25}(10)^{(1, \le 2), (4, \ge 2)}$ into
$\mathcal{M}_{25}(10)^{(1, = 1), (4, \ge 2)}$ and $\mathcal{M}_{25}(10)^{(1, = 2), (4, \ge 2)}$ because
\begin{align*}
|\mathcal{M}_{25}(10)^{(1, = 1), (4, \ge 2)}|&\ =\ 10\ =\ |\mathcal{M}_{23}(8)^{(1, \ge 3)}| \mbox{ and }\\
|\mathcal{M}_{25}(10)^{(1, = 2), (4, \ge 2)}|&\ =\ 4\ =\ |\mathcal{M}_{23}(9)^{(1, \ge 3)}|.
\end{align*}
We, therefore, define the two bijections: 
\begin{itemize}
    \item from $\mathcal{M}_{23}(8)^{(1, \ge 3)}$ to $\mathcal{M}_{25}(10)^{(1, = 1), (4, \ge 2)}$:
    \begin{align*}F\ \mapsto\ &\{\min F+2, \min F+3\}\cup \left(\{a\in F\,:\, \min_2 F\le a\le \min_{3} F\} + 1\right) \\ 
    &\cup \{\min_4 F + 2, \max_4 F + 2\}\\
    &\cup \left(\{a\in F\,:\, \max_{3} F\le a\le \max_2 F\}+3\right)\cup \{\max F + 1, \max F + 2\}\mbox{ and}\end{align*}
    \item from $\mathcal{M}_{23}(9)^{(1, \ge 3)}$ to 
$\mathcal{M}_{25}(10)^{(1, = 2), (4, \ge 2)}$:
\begin{align*}F\ \mapsto\ &\{\min F+1, \min F+3\}\cup \left(\{a\in F\,:\, \min_2 F\le a\le \min_{3} F\}+1\right) \\ 
    &\cup \{\min_4 F + 2, \max_4 F +1\}\\
    &\cup \left(\{a\in F\,:\, \max_{3} F\le a\le \max_2 F\}+2\right)\cup \{\max F, \max F + 2\}.\end{align*}
\end{itemize}  
For general $n$ and even $k\ge 6$, we have the maps:
\begin{align*}
&\Delta_{n,k}: \mathcal{M}_{2n+3}(k)^{(1, \ge 3)} \rightarrow \mathcal{M}_{2n+5}(k+2)^{(1, = 1), \left(\frac{k}{2}, \ge 2\right)},\\
    &F\ \mapsto\ \{\min F+2, \min F+3\}\cup(\{a\in F\,:\, \min_2 F\le a \le \min_{k/2-1} F\}+1)\\
    &\cup\{\min_{k/2} F+2, \max_{k/2} F+2\}\\
    &\cup (\{a\in F\,:\, \max_{k/2-1} F\le a \le \max_{2} F\}+3)\cup\{\max F + 1,\max F+2\}
    \end{align*}
and 
\begin{align*}
    &\xi_{n,k}: \mathcal{M}_{2n+3}(k+1)^{(1, \ge 3)} \rightarrow \mathcal{M}_{2n+5}(k+2)^{(1, = 2), \left(\frac{k}{2}, \ge 2\right)},\\
    &F\ \mapsto\ \{\min F+1, \min F+3\}\cup(\{a\in F\,:\,\min_2 F\le a\le  \min_{k/2-1} F\}+1)\\
    &\cup\{\min_{k/2}F+2, \max_{k/2} F+1\}\\
    &\cup(\{a\in F\,:\,\max_{k/2-1} F\le a\le  \max_{2} F\}+2)\cup\{\max F, \max F+2\}.
\end{align*}

\subsubsection{All the bijective maps suggested by the data}
Recall that we want to devise bijective maps that show $|\mathcal{M}_{2n+1}\cup \mathcal{M}_{2n+3}| = |\mathcal{M}_{2n+5}|$. Our data analysis in Subsection \ref{data_analysis}
produces the following maps:
\begin{itemize}
    \item $\Psi_n: \mathcal{M}_{2n+1}\rightarrow\mathcal{M}^{(2)}_{2n+5} = \mathcal{M}^{(1, \ge 3)}_{2n+5}$,
$$F\ \mapsto\ \{\min F\}\cup \left(\{a\in F\,:\, \min F < a < 2n+1\}+2\right)\cup \{2n+5\},$$
\item $\Theta_{n,k}: \mathcal{M}_{2n+3}(k)^{(1, \le 2)}\ \rightarrow\ \mathcal{M}_{2n+5}(k+1)^{(1,\le 2)}$,
$$F\ \mapsto\ (\{a\in F: a < s_F\}+1)\cup \{s_F + 3/2\}\cup (\{a\in F: a > s_F\}+2),$$
\item $\Gamma_{n,k}:     \mathcal{M}_{2n+5}(k+2)^{(1, \le 2), \left(\frac{k}{2}, =1\right)}\ \rightarrow\ \mathcal{M}_{2n+3}(k+1)^{(1, \le 2)}$,
$$F\ \mapsto\ (\{a\in F: a\le \min_{k/2} F\} - 1)\cup\{(s_F-3)/2\}\cup (\{a\in F: a \ge \max_{k/2} F\}-2),$$
\item $\Delta_{n,k}: \mathcal{M}_{2n+3}(k)^{(1, \ge 3)} \rightarrow \mathcal{M}_{2n+5}(k+2)^{(1, = 1), \left(\frac{k}{2}, \ge 2\right)}$,
\begin{align*}
    &F\ \mapsto\ \{\min F+2, \min F+3\}\cup(\{a\in F\,:\, \min_2 F\le a \le \min_{k/2-1} F\}+1)\\
    &\cup\{\min_{k/2} F+2, \max_{k/2} F+2\}\\
    &\cup (\{a\in F\,:\, \max_{k/2-1} F\le a \le \max_{2} F\}+3)\cup\{\max F + 1,\max F+2\},
    \end{align*}
\item $\xi_{n,k}: \mathcal{M}_{2n+3}(k+1)^{(1, \ge 3)} \rightarrow \mathcal{M}_{2n+5}(k+2)^{(1, = 2), \left(\frac{k}{2}, \ge 2\right)}$,
\begin{align*}
    &F\ \mapsto\ \{\min F+1, \min F+3\}\cup(\{a\in F\,:\,\min_2 F\le a\le  \min_{k/2-1} F\}+1)\\
    &\cup\{\min_{k/2}F+2, \max_{k/2} F+1\}\\
    &\cup(\{a\in F\,:\,\max_{k/2-1} F\le a\le  \max_{2} F\}+2)\cup\{\max F, \max F+2\}.
\end{align*}
\end{itemize}
We prove that all these maps are bijections in the next subsection. If they were, we would have 
\begin{align*}&|\mathcal{M}_{2n+1}| + \sum_{k=6, 2|k}^\infty|\mathcal{M}_{2n+3}(k)^{(1, \le 2)}| + \sum_{k=6, 2|k}^\infty |\mathcal{M}_{2n+3}(k+1)^{(1, \le 2)}|\\
&+ \sum_{k=6,2|k}^\infty|\mathcal{M}_{2n+3}(k)^{(1, \ge 3)}| + \sum_{k=6, 2|k}^\infty|\mathcal{M}_{2n+3}(k+1)^{(1, \ge 3)}| \\
&\ =\ |\mathcal{M}^{(1, \ge 3)}_{2n+5}|+\sum_{k=6, 2|k}^\infty|\mathcal{M}_{2n+5}(k+1)^{(1,\le 2)}| + \sum_{k=6, 2|k}^\infty|\mathcal{M}_{2n+5}(k+2)^{(1, \le 2), \left(\frac{k}{2}, =1\right)}|\\
&+ \sum_{k=6,2|k}^\infty|\mathcal{M}_{2n+5}(k+2)^{(1, = 1), (\frac{k}{2}, \ge 2)}| + \sum_{k=6, 2|k}^\infty|\mathcal{M}_{2n+5}(k+2)^{(1, = 2), (\frac{k}{2}, \ge 2)}|.\end{align*}
It follows that
\begin{align*}
&|\mathcal{M}_{2n+1}| + \sum_{k=6}^\infty |\mathcal{M}_{2n+3}(k)^{(1,\le 2)}| + \sum_{k=6}^\infty |\mathcal{M}_{2n+3}(k)^{(1,\ge 3)}|\\
&\ =\ |\mathcal{M}^{(1, \ge 3)}_{2n+5}| + \sum_{k=6, 2|k}^\infty|\mathcal{M}_{2n+5}(k+1)^{(1, \le 2)}| + \sum_{k=6, 2|k}^\infty |\mathcal{M}_{2n+5}(k+2)^{(1, \le 2)}|.
\end{align*}
Therefore,
$$|\mathcal{M}_{2n+1}| + \sum_{k=6}^\infty |\mathcal{M}_{2n+3}(k)|\ =\ |\mathcal{M}^{(1, \ge 3)}_{2n+5}| + \sum_{k=7}^\infty|\mathcal{M}_{2n+5}(k)^{(1,\le 2)}|.$$

To prove $|\mathcal{M}_{2n+1}\cup \mathcal{M}_{2n+3}| = |\mathcal{M}_{2n+5}|$, it remains to verify that
\begin{equation}\label{e65}\sum_{k=2}^5|\mathcal{M}_{2n+3}(k)|\ =\ \sum_{k=4}^6|\mathcal{M}_{2n+5}(k)^{(1, \le 2)}|.\end{equation}
The lower bound for $k$ on the right side of \eqref{e65} is $4$ because if $k\le 3$, then for $F\in\mathcal{M}_{2n+5}(k)^{(1, \le 2)}$, we witness
$$\min_2 F\ \le\ \min F + 2\ =\ k + 2\ \le 5\ <\ 2n+3 \ \le\ \max_2 F,$$
which contradicts $|F| = k\le 3$.

\begin{proof}[Proof of \eqref{e65}]
    For $n = 2$,
$$\mathcal{M}_{7}(2) \ =\ \{\{2, 7\}\}, \mathcal{M}_7(3)\ =\ \{\{3, 5, 7\}\}, \mathcal{M}_7(4)\ =\ \{\{4, 5, 6, 7\}\},\mathcal{M}_{7}(5)\ =\ \emptyset,$$
and
$$        
\mathcal{M}^{(1)}_9(4)\ =\ \{\{4, 5, 8, 9\}, \{4, 6, 7, 9\}\}, \mathcal{M}^{(1)}_9(5) \ =\ \{\{5, 6, 7, 8, 9\}\}, \mathcal{M}^{(1)}_9(6)\ =\ \emptyset,
$$
so
\eqref{e65} holds for $n = 2$. We assume $n\ge 3$.

We have
\begin{align*}
    \mathcal{M}_{2n+5}(4)^{(1, \le 2)}&\ =\ \{\{4, 5, 2n+4, 2n+5\}, \{4, 6, 2n+3, 2n+5\}\},\\
    \mathcal{M}_{2n+5}(5)^{(1, \le 2)}&\ =\ \{\{5, 6, n+5, 2n+4, 2n+5\}, \{5, 7, n+5, 2n+3, 2n+5\}\},\mbox{ and}\\
    \mathcal{M}_{2n+5}(6)^{(1, \le 2)}&\ =\ \{\{6, 7, i, 2n+11-i, 2n+4, 2n+5\}\,:\, 8\le i\le n+5 \}\\
    &\ \cup\ \{\{6, 8, i, 2n+11-i, 2n+3, 2n+5\}\,:\, 9\le i\le n+5\}.
\end{align*}
Hence, 
\begin{equation}\label{e67}|\mathcal{M}_{2n+5}(4)^{(1, \le 2)}|+|\mathcal{M}_{2n+5}(5)^{(1, \le 2)}|+|\mathcal{M}_{2n+5}(6)^{(1,\le 2)}|\ =\ 2 + 2 + (2n-5) \ =\ 2n-1.\end{equation}

Meanwhile, 
\begin{align*}
    \mathcal{M}_{2n+3}(2)&\ =\ \{\{2, 2n+3\}\}, \quad \mathcal{M}_{2n+3}(3)\ =\ \{\{3, n+3, 2n+3\}\},\\
    \mathcal{M}_{2n+3}(4)&\ =\ \{\{4, i, 2n+7-i, 2n+3\}\,:\, 5\le i\le n+3\}, \mbox{ and}\\
    \mathcal{M}_{2n+3}(5)&\ =\ \{\{5, i, n+4, 2n+8-i, 2n+3\}\,:\, 6\le i\le n+3\}.
\end{align*}
Hence, 
\begin{align}\label{e68}|\mathcal{M}_{2n+3}(2)|+|\mathcal{M}_{2n+3}(3)|+|\mathcal{M}_{2n+3}(4)| + |\mathcal{M}_{2n+3}(5)|&\ =\ 1 + 1 + (n-1) +(n-2)\nonumber\\
&\ =\ 2n-1.\end{align}
Due to \eqref{e67} and \eqref{e68}, we have proved \eqref{e65}.
\end{proof}

\subsubsection{Verification of bijective maps}
We show that $\Psi_n$, $\Theta_{n,k}$, $\Gamma_{n,k}$, and $\Delta_{n,k}$ are bijective. Bijectivity of $\xi_{n,k}$ is moved to Appendix \ref{bi-xi} because the proof is similar to the proof for $\Delta_{n,k}$. 
\begin{lem}\label{kp1}
   For each $n\ge 2$, the map $\Psi_n: \mathcal{M}_{2n+1}\rightarrow \mathcal{M}^{(2)}_{2n+5}$ is a bijection. Consequently, we have $|\mathcal{M}_{2n+1}| = |\mathcal{M}^{(2)}_{2n+5}|$.
\end{lem}

\begin{proof}
First, we show that $\Psi_n: \mathcal{M}_{2n+1}\rightarrow\mathcal{M}^{(2)}_{2n+5} = \mathcal{M}^{(1, \ge 3)}_{2n+5}$,
$$F\ \mapsto\ \{\min F\}\cup \left(\{a\in F\,:\, \min F < a < 2n+1\}+2\right)\cup \{2n+5\}$$
is well-defined. Let $F\in \mathcal{M}_{2n+1}$. Since $2n+1\ge 5$, the set $F$ cannot be $\{2n+1\}$; that is, $F$ must have at least $2$ elements, and $\min_2 F- \min F\ge 1$. We confirm that
$$\min_2\Psi_n(F) - \min \Psi_n(F)\ \ge\ (\min_2 F + 2) - \min F \ \ge\ 3,\quad 2n+5\in \Psi_n(F), \mbox{ and}$$
$$\min \Psi_n(F) \ =\ \min F\ =\ |F|\ =\ |\Psi_n(F)|.$$
Furthermore, the symmetry of $F$ guarantees the symmetry of $\Psi_n(F)$.
Hence, we have $\Psi_n(F)\in \mathcal{M}^{(2)}_{2n+5}$, and $\Psi_n$ is well-defined.

Next, it follows from the definition that $\Psi_n$ is injective. We verify that $\Psi_n$ is surjective. Pick $E\subset [2n+5]$ such that $E$ is symmetric, $\min E = |E|$, $2n+5\in E$, and $\min_2 E - \min E\ge 3$. Define
$$F \ :=\ (E\backslash \{\min E, 2n+5\}-2)\cup \{\min E, 2n+1\}.$$
Since $E$ is symmetric, it follows from $\min_2 E - \min E\ge 3$ that $\max_2 E \le 2n+2$, so $\max F = 2n+1$. Furthermore, we have
$$\min F \ =\ \min E\ =\ |E| \ =\ |F|,$$
and due to the symmetry of $E$,
\begin{align*}
&\max F + \min F - F\\
&\ =\ 2n+1+\min E - ((E\backslash \{\min E, 2n+5\}-2)\cup \{\min E, 2n+1\})\\
&\ =\ (2n+1+\min E - (E\backslash\{\min E, 2n+5\}-2))\cup  \{\min E, 2n+1\}\\
&\ =\ (2n+3+\min E - E\backslash\{\min E, 2n+5\})\cup \{\min E, 2n+1\}\\
&\ =\ ((2n+5+\min E - E\backslash\{\min E, 2n+5\})-2)\cup \{\min E, 2n+1\}\\
&\ =\ (E\backslash \{\min E, 2n+5\} - 2)\cup \{\min E, 2n+1\}\ =\ F.
\end{align*}
Hence, the set $F$ is symmetric, and thus, $F\in \mathcal{M}_{2n+1}$. By the definition of $\Psi_n$, we have $\Psi_n(F) = E$. Therefore, $\Psi_n$ is indeed surjective.
\end{proof}

\begin{lem}
    For $n\ge 2$ and even $k\ge 6$, the map $\Theta_{n,k}$ is a bijection. As a result, 
    $$|\mathcal{M}_{2n+3}(k)^{(1, \le 2)}|\ =\ |\mathcal{M}_{2n+5}(k+1)^{(1, \le 2)}|.$$
\end{lem}

\begin{proof}
    First, we verify that $\Theta_{n,k}: \mathcal{M}_{2n+3}(k)^{(1, \le 2)}\ \rightarrow\ \mathcal{M}_{2n+5}(k+1)^{(1,\le 2)}$,
$$F\ \mapsto\ (\{a\in F: a < s_F\}+1)\cup \{s_F + 3/2\}\cup (\{a\in F: a > s_F\}+2)$$
    is well-defined. Pick $F\in \mathcal{M}_{2n+3}(k)^{(1,\le 2)}$. Since $k\ge 6$, the set
    $\{a\in F: a < s_F\}$ has at least 3 elements. Hence, 
    $$\min_2 \Theta_{n,k}(F) - \min \Theta_{n,k}(F)\ =\ (\min_2 F+1) - (\min F + 1)\ =\ \min_2 F - \min F\ \le\ 2.$$
    Since $\max F = 2n+3$, we have $\max \Theta_{n,k}(F) = 2n+5$. Furthermore,
    $$\min \Theta_{n,k}(F) \ =\ k+1\ =\ |F|+1 \ =\ \Theta_{n,k}(F).$$
    Finally, the set $\Theta_{n,k}(F)$ is symmetric due to the symmetry of $F$. Hence, $\Theta_{n,k}(F)$ is in $\mathcal{M}_{2n+5}(k+1)^{(1,\le 2)}$.
    
    Clearly, the map $\Theta_{n,k}$ is injective by definition. We need to show that $\Theta_{n,k}$ is surjective. Pick $E\in \mathcal{M}_{2n+5}(k+1)^{(1, \le 2)}$. Let 
    $$F\ :=\ \left(\left\{a\in E: a < s_E\right\}-1\right)\cup \left(\left\{a\in E: a > s_E\right\}-2\right).$$
    We have
    \begin{itemize}
        \item $\min_2 F - \min F = (\min_2 E - 1) - (\min E - 1) = \min_2 E - \min E\le 2$, 
         \item $\max F = 2n+3$ because $\max E = 2n+5$,
        \item $\min F = \min E - 1 = k = |E| - 1 = |F|$,  and
        \item due to the symmetry of $E$, the set $F$ is also symmetric. 
    \end{itemize}
    Hence, the set $F$ is in $\mathcal{M}_{2n+3}(k)^{(1, \le 2)}$. Last but not least,
    $$s_F\ =\ \frac{\min F + \max F}{2}\ =\ \frac{\min E + \max E - 3}{2}\ =\ s_E - \frac{3}{2},$$
    so 
    \begin{align*}
    &\Theta_{n,k}(F)\\
    &\ =\ (\{a\in F: a < s_F\}+1)\cup \{s_F + 3/2\}\cup (\{a\in F: a > s_F\}+2)\\
    &\ =\ \left(\left\{a\in F: a \le s_E-2\right\}+1\right)\cup \{s_E\}\cup \left(\left\{a\in F: a \ge s_E-1\right\}+2\right)\\
    &\ =\ \{a\in E: a < s_E\}\cup \{s_E\} \cup \{a\in E: a > s_E\}\ =\ E.
    \end{align*}
    Therefore, the map $\Theta_{n,k}$ is bijective. 
\end{proof}

\begin{lem}
 For $n\ge 2$ and even $k\ge 6$, the map $\Gamma_{n,k}$ is a bijection. As a result,
 $$|\mathcal{M}_{2n+5}(k+2)^{(1, \le 2), \left(\frac{k}{2}, =1\right)}|\ =\ |\mathcal{M}_{2n+3}(k+1)^{(1, \le 2)}|.$$
\end{lem}
\begin{proof}
Consider $\Gamma_{n,k}:     \mathcal{M}_{2n+5}(k+2)^{(1, \le 2), \left(\frac{k}{2}, =1\right)}\ \rightarrow\ \mathcal{M}_{2n+3}(k+1)^{(1, \le 2)}$,
$$F\ \mapsto\ (\{a\in F: a\le \min_{k/2} F\} - 1)\cup\{(s_F-3)/2\}\cup (\{a\in F: a \ge \max_{k/2} F\}-2).$$
Let $F\in \mathcal{M}_{2n+5}(k+2)^{(1, \le 2), \left(\frac{k}{2}, =1\right)}$.
We have
\begin{itemize}
    \item $\max \Gamma_{n,k}(F) = \max F -2 = 2n+5-2 = 2n+3$,
    \item $\min \Gamma_{n,k}(F) = \min F - 1 = |F| - 1 = |\Gamma_{n,k}(F)|$, so $|\Gamma_{n,k}(F)| = k+1$, 
    \item $\min_2 \Gamma_{n,k}(F) - \min \Gamma_{n,k}(F) = (\min_2 F - 1) - (\min F-1) = \min_2 F - \min F \le 2$, and
    \item due to the symmetry of $F$, the set $\Gamma_{n,k}(F)$ is also symmetric.
\end{itemize}
Therefore, $\Gamma_{n,k}$ is well-defined. 

By the definition of $\Gamma_{n,k}$ and the condition that $\min_{k/2+1} F - \min_{k/2} F = 1$ for all $F$ in the domain $\mathcal{M}_{2n+5}(k+2)^{(1, \le 2), \left(\frac{k}{2}, =1\right)}$, the injectivity of $\Gamma_{n,k}$ is guaranteed. We show that $\Gamma_{n,k}$ is surjective. 
Pick $E\in \mathcal{M}_{2n+3}(k+1)^{(1, \le 2)}$. Let
$$F\ :=\ (\{a\in E\,:\, a < s_E\}+1)\cup \{\min_{k/2} E+2, \max_{k/2}E+1\}\cup (\{a\in E\,:\, a > s_E\}+2).$$
It is routine to check that $F\in \mathcal{M}_{2n+5}(k+2)^{(1, \le 2), \left(\frac{k}{2}, =1\right)}$, and furthermore, $\Gamma_{n,k}(F) = E$. This concludes our proof that $\Gamma_{n,k}$ is a bijection. 
\end{proof}

\begin{lem}
For $n\ge 2$ and even $k\ge 6$, the map $\Delta_{n,k}$ is a bijection. As a result,
$$|\mathcal{M}_{2n+3}(k)^{(1, \ge 3)}| \ =\ |\mathcal{M}_{2n+5}(k+2)^{(1, = 1), \left(\frac{k}{2}, \ge 2\right)}|.$$
\end{lem}

\begin{proof} 
Recall the map
$\Delta_{n,k}: \mathcal{M}_{2n+3}(k)^{(1, \ge 3)} \rightarrow \mathcal{M}_{2n+5}(k+2)^{(1, = 1), \left(\frac{k}{2}, \ge 2\right)}$,
\begin{align*}
    &F\ \mapsto\ \{\min F+2, \min F+3\}\cup(\{a\in F\,:\, \min_2 F\le a \le \min_{k/2-1} F\}+1)\\
    &\cup\{\min_{k/2} F+2, \max_{k/2} F+2\}\\
    &\cup (\{a\in F\,:\, \max_{k/2-1} F\le a \le \max_{2} F\}+3)\cup\{\max F + 1,\max F+2\}.
    \end{align*}
We have 
\begin{align*}
&\min_2\Delta_{n,k}(F)-\min\Delta_{n,k}(F)\ =\ (\min F+3)-(\min F + 2)\ =\ 1,\\
&\min_{k/2+1}\Delta_{n,k}(F)-\min_{k/2}\Delta_{n,k}(F)\ =\ (\min_{k/2}F+2)-(\min_{k/2-1}F+1)\ \ge\ 2,\\
&\max \Delta_{n,k}(F) \ =\ \max F + 2\ =\ 2n+5,\\
&\min \Delta_{n,k}(F)\ =\ k+2\ =\ |F| + 2\ =\ |\Delta_{n,k}(F)|,
\end{align*}
and $\Delta_{n,k}(F)$ is symmetric because $F$ is symmetric. Hence, $\Delta_{n,k}$ is well-defined. 

From the definition, $\Delta_{n,k}$ is clearly injective. We show that $\Delta_{n,k}$ is surjective. Pick $E\in \mathcal{M}_{2n+5}(k+2)^{(1, = 1), \left(\frac{k}{2}, \ge 2\right)}$. Let 
\begin{align*}
F\ :=\ \{\min_2 E-3\}&\cup(\{a\in E\,:\, \min_{3}E\le a\le\min_{k/2}E\}-1)\\
&\cup\{\min_{k/2+1}E-2, \max_{k/2+1}E-2\}\\
&\cup(\{a\in E\,:\, \max_{k/2}E\le a\le \max_3 E\}-3)\cup\{\max_2 E - 1\}
\end{align*}
We verify that
\begin{align*}
    &\min_2 F-\min F\ =\ (\min_3 E - 1) - (\min_2 E-3)\ \ge\ 3,\\
    &\max F\ =\ \max_2 E - 1\ =\ (2n+4) - 1\ =\ 2n+3,\\
    &\min F\ =\ \min_2 E-3\ =\ (k+3)-3\ =\ k\ =\ |F|,
\end{align*}
and $F$ is symmetric because $E$ is symmetric. Hence, $F\in \mathcal{M}_{2n+3}(k)^{(1,\ge 3)}$. Since $\Delta(F) = E$, the map $\Delta_{n,k}$ is surjective. 
\end{proof}

\section{Schreier sets that do not contain a given integer}\label{sec-nointeger}

\subsection{Data analysis and proof ideas}

Table \ref{Data_oneint} gives us the initial terms of the sequences $(|\mathcal{F}_{u,n}|)_{n=1}^\infty$ with $1\le u\le 5$. While the sequences $(|\mathcal{F}_{1,n}|)_{n=1}^\infty$ and $(|\mathcal{F}_{2,n}|)_{n=1}^\infty$ are simple, latter sequences seem mysterious. Suppose we want to guess a linear recurrence for the sequence $(|\mathcal{F}_{3,n}|)_{n=1}^\infty$ based on its initial terms. If $(|\mathcal{F}_{3,n}|)_{n=1}^\infty$ satisfied a linear recurrence of order $5$, i.e., there would exist $(c_i)_{i=1}^5$ such that
$$c_5|\mathcal{F}_{3, n-5}| + c_4|\mathcal{F}_{3,n-4}| + \cdots + c_1|\mathcal{F}_{3,n-1}|\ =\ |\mathcal{F}_{3,n}|, \mbox{ for all large }n,$$ 
then $(c_i)_{i=1}^5$ is a solution to the system of linear equations:
$$\begin{cases}
    2c_5 + 3c_4 + 5c_3 + 8c_2 + 12c_1&\ =\ 17,\\
    3c_5 + 5c_4 + 8c_3 + 12c_2 + 17c_1 &\ =\ 23,\\
    5c_5 + 8c_4 + 12c_3 + 17c_2 + 23c_1&\ =\ 30,\\
    8c_5 + 12c_4 + 17c_3 + 23c_2 + 30c_1 &\ =\ 38,\\
    12c_5 + 17c_4 + 23c_3 + 30c_2 + 38c_1&\ =\ 47.
\end{cases}$$
Solving the system, we obtain 
$$\begin{cases}c_5 &\ =\ 6-c_2 - 3c_1\\ c_4&\ =\ -15 + 3c_2 + 8c_1\\ c_3 &\ =\ 10-3c_2-6c_1\\ c_1, c_2\mbox{ free}.\end{cases}$$
Due to the two free variables, we adjust our guess accordingly: $(|\mathcal{F}_{3,n}|)_{n=1}^\infty$ may satisfy a linear recurrence of order $3$. We then solve the system
$$\begin{cases}
    2c_3 + 3c_2 + 5c_1 &\ =\ 8,\\
    3c_3 + 5c_2 + 8c_1 &\ =\ 12,\\
    5c_3 + 8c_2 + 12c_1 &\ =\ 17
\end{cases}$$
to obtain $(c_1, c_2, c_3) = (3, -3, 1)$. Hence, our conjectured recurrence for $(|\mathcal{F}_{3,n}|)_{n=1}^\infty$ is 
$$|\mathcal{F}_{3,n-3}|- 3|\mathcal{F}_{3,n-2}|+3|\mathcal{F}_{3,n-1}|\ =\ |\mathcal{F}_{3, n}|.$$
Latter terms of $(|\mathcal{F}_{3,n}|)_{n=1}^\infty$ support this conjecture: for instance,
$47-3\cdot 57 + 3\cdot 68 = 80$.

We note from Table \ref{Data_oneint} that
\begin{align*}
|\mathcal{F}_{1,n}|&\ =\ |\mathcal{F}_{1,n-1}|,\\
|\mathcal{F}_{2,n}|&\ =\ 2|\mathcal{F}_{2,n-1}| - |\mathcal{F}_{2,n-2}|,\mbox{ and}\\
|\mathcal{F}_{3,n}|&\ =\ 3|\mathcal{F}_{3,n-1}| - 3|\mathcal{F}_{3,n-2}| + |\mathcal{F}_{3,n-3}|.
\end{align*}
These lead us to Theorem \ref{m2}, which states that for $u\in\mathbb{N}$:
\begin{equation}\label{e90}|\mathcal{F}_{u,n}|\ =\ \sum_{i=1}^u (-1)^{i-1}\binom{u}{i}|\mathcal{F}_{u,n-i}|,\mbox{ whenever }n\ge 2u.\end{equation}

The alternating signs in Recurrence \eqref{e90} naturally suggest a proof based on the inclusion-exclusion principle. Another promising approach is induction on $u$, which would require establishing a relationship between consecutive rows of Table \ref{Data_oneint}. Specifically, our goal is to express each entry in a given row in terms of entries from the row immediately below it (see Lemma \ref{l100}). The next two subsections are devoted to these two different proofs of \eqref{e90}. Interested readers may refer to \cite[Theorem 1.1]{BCF} for an earlier application of the inclusion-exclusion principle to derive a linear recurrence with alternating signs, and to \cite{BGHH, C2, C3} for inductive proofs. Recently, \cite{CV} presented an alternative proof of \cite[Theorem 1.1]{BCF} using the characteristic polynomial method. 

We end this subsection by proving the first statement of Theorem \ref{m2}: for $u\le n\le 2u-1$, we have $|\mathcal{F}_{u,n}| = F_n$. 

\begin{proof}[Proof that $|\mathcal{F}_{u,n}| = F_n$ for $u\le n\le 2u-1$]
    For $k\in \mathbb{N}$, we count the number of $k$-element sets $F$ in $\mathcal{F}_{u,n}$. The Schreier condition requires that $F$ is a subset of $\{k, k+1, \ldots, n\}$. Since $F$ contains $u$, we must have $k\le u$. It follows that each $F$ is uniquely determined by  $k-1$ integers from $\{k, k+1, \ldots, n\}\backslash\{u\}$. Hence, there are
    $\binom{n-k}{k-1}$ such sets $F$, and thus,
    \begin{align*}|\mathcal{F}_{u,n}| \ =\ \sum_{k=1}^{\min\left\{u,\left\lfloor\frac{n+1}{2}\right\rfloor\right\}}\binom{n-k}{k-1}&\ =\ \sum_{k=1}^{\left\lfloor\frac{n+1}{2}\right\rfloor}\binom{n-k}{k-1}\quad\mbox{ because }n\le 2u-1\\
    &\ =\ \sum_{k=0}^{\left\lfloor\frac{n-1}{2}\right\rfloor}\binom{n-1-k}{k}\ =\ F_n\quad\mbox{ by }\eqref{oot}.
    \end{align*}
\end{proof}

\subsection{Proof of \eqref{e90} using the inclusion-exclusion principle}
The following lemma is an analog of \cite[Lemma 2.1]{BCF}.
\begin{lem}\label{iel}
    Fix $u\in \mathbb{N}$ and $n\ge 2u$. Let $G\subset\{n-u+1, \ldots, n-1, n\}$ be nonempty and set 
    $$\mathcal{M}_G\ :=\ \{F\in \mathcal{F}_{u,n}\,:\, F\cap G = \emptyset\}.$$
    Then $|\mathcal{M}_G| = |\mathcal{F}_{u, n-|G|}|$.
\end{lem}

\begin{proof}
    Let $\eta: [n]\backslash G\rightarrow [n-|G|]$ be the unique increasing bijection. Define the map $\gamma: \mathcal{M}_G\rightarrow \mathcal{F}_{u, n-|G|}$ as
    $\gamma(F) = \eta(F)$.

    First, we show that $\gamma$ is well-defined. Let $F \in \mathcal{M}_G$. Since $u\in F$ and $u < \min G$, we have $\eta(u) = u$, so $u\in \gamma(F)$. If $\min F < \min G$, then $\min \gamma(F) = \min F$, which gives
    $$|\gamma(F)| \ =\ |F| \ \le\ \min F\ =\ \min \gamma(F).$$
    If $\min F > \min G$, then $\min \gamma(F)\ge n-u+1$ and $|F|\le u-1$ because $G\subset\{n-u+1, \ldots, n-1, n\}$ is nonempty and $F\cap G = \emptyset$. In this case, we have
    $$|\gamma(F)|\ =\ |F|\ \le\ u-1\ <\ n-u+1\ \le\ \min \gamma(F).$$
    Therefore, the set $\gamma(F)$ is Schreier and thus, is in $\mathcal{F}_{u, n-|G|}$.

    Injectivity is obvious from the definition of $\gamma$. Let us show surjectivity. Pick 
    $E\in \mathcal{F}_{u, n-|G|}$ and let $F = \gamma^{-1}(E)$. Then $F\cap G = \emptyset$ and $\min F\ge \min E$, so 
    $$|F| \ =\ |E| \ \le\ \min E\ \le\ \min F.$$
    Furthermore, since $u\in E$ and $u < \min G$, we have $u = \eta^{-1}(u)\in F$. Therefore, the set $F$ is in $\mathcal{M}_G$.
\end{proof}

\begin{proof}[Proof of Theorem \ref{m2} using the inclusion-exclusion principle]
    Using the notation from Lemma \ref{iel}, we have that the set
    $$\mathcal{F}_{u,n}\backslash\bigcup_{i=0}^{u-1} \mathcal{M}_{\{n-i\}}\ =\ \{F\in \mathcal{F}_{u,n}\,:\, \{n-u+1, \ldots, n-1, n\}\subset F\},\mbox{ with }n\ge 2u,$$
    is empty because $u\in F$ and $\{n-u+1, \ldots, n-1, n\}\subset F$ imply that
    $$|F|\ \ge\ u+1 \ >\ u\ \ge\ \min F,$$
    violating the Schreier condition. By the inclusion-exclusion principle, we obtain
    \begin{align*}
        0 &\ =\ \left|\mathcal{F}_{u,n}\backslash\bigcup_{i=0}^{u-1} \mathcal{M}_{\{n-i\}}\right|\\
        &\ =\ |\mathcal{F}_{u,n}| - \sum_{i=1}^u (-1)^{i-1}\sum_{G\in \mathcal{G}_i}|\mathcal{M}_{G}|,
    \end{align*}
    where $\mathcal{G}_i = \{G\subset \{n-u+1, \ldots, n-1, n\}: |G| = i\}$. By Lemma \ref{iel}, we have $|\mathcal{M}_{G}| = |\mathcal{F}_{u, n-|G|}|$ and thus, obtain
    \begin{align*}
    |\mathcal{F}_{u,n}| \ =\ \sum_{i=1}^u (-1)^{i-1} \sum_{G\in \mathcal{G}_i}|\mathcal{M}_{G}|
    &\ =\ \sum_{i=1}^u (-1)^{i-1}|\mathcal{G}_i| |\mathcal{F}_{u, n-i}|\\
    &\ =\ \sum_{i=1}^u (-1)^{i-1}\binom{u}{i}|\mathcal{F}_{u, n-i}|,
    \end{align*}
    as desired.
\end{proof}

\subsection{Proof of \eqref{e90} using induction}

The next lemma, which is used in the inductive step, connects two consecutive rows of Table \ref{Data_oneint}.

\begin{lem}\label{l100} For $u\ge 2$ and $n\ge u+1$, we have
    $$|\mathcal{F}_{u,n}| - |\mathcal{F}_{u,n-1}| \ =\ |\mathcal{F}_{u-1, n-2}|.$$
\end{lem}

\begin{proof}
Observe that
$\mathcal{F}_{u, n-1}\subset \mathcal{F}_{u, n}$ and
$$\mathcal{F}_{u,n}\backslash \mathcal{F}_{u,n-1}\ =\ \{F\subset[n]\,:\,F\mbox{ is Schreier and } u, n\in F\}.$$
Define the map $\gamma: \mathcal{F}_{u-1, n-2}\rightarrow\mathcal{F}_{u,n}\backslash \mathcal{F}_{u,n-1}$ as $\gamma(F) = (F+1)\cup \{n\}$. Then 
$$u\in \gamma(F)\mbox{ and }\min \gamma(F) \ =\ \min F + 1 \ \ge\ |F| + 1\ =\ |\gamma(F)|,$$
so $\gamma(F)\in \mathcal{F}_{u,n}\backslash \mathcal{F}_{u,n-1}$, and $\gamma$ is well-defined. 

Injectivity is immediate from the definition of $\gamma$. We show surjectivity. Let $E\in \mathcal{F}_{u,n}\backslash \mathcal{F}_{u,n-1}$ and let
$$F\ :=\ \gamma^{-1}(E)\ =\ (E\backslash \{n\})-1.$$
Then 
$$u-1\in F, \max F \ \le\ (n-1)-1 \ =\ n-2,\mbox{ and }\min F\ =\ \min E - 1\ \ge\ |E| - 1\ =\ |F|.$$
Hence, the set $F$ is in $\mathcal{F}_{u-1, n-2}$. That $\gamma(F) = E$ shows that $\gamma$ is surjective. 
\end{proof}

\begin{proof}[Proof of Theorem \ref{m2} using induction]
We wish to prove the following statement for each $u\ge 1$:
$$\mbox{Statement }P(u):\quad\quad\quad \mbox{For all }n\ge 2u, \mbox{ we have }|\mathcal{F}_{u,n}| \ =\ \sum_{i=1}^u (-1)^{i-1}\binom{u}{i}|\mathcal{F}_{u, n-i}|.$$
We proceed by induction on $u$. 

Base case: when $u = 1$, the statement $P(1)$ states that for all $n\ge 2$, we have
$|\mathcal{F}_{1,n}| = |\mathcal{F}_{1, n-1}|$, which is true because 
$\mathcal{F}_{1,n} = \{1\}$  for every $n\in \mathbb{N}$.

Inductive hypothesis: assume that $P(u)$ is true for some $u\ge 1$, i.e., 
    $$|\mathcal{F}_{u,n}| - \sum_{i=1}^u (-1)^{i-1}\binom{u}{i}|\mathcal{F}_{u, n-i}|\ =\ 0,\mbox{ for all }n\ge 2u.$$
    For each $n\ge 2u$, by Lemma \ref{l100}, we have
    $$|\mathcal{F}_{u+1,n+2}|-|\mathcal{F}_{u+1,n+1}| - \sum_{i=1}^u (-1)^{i-1}\binom{u}{i}(|\mathcal{F}_{u+1, n-i+2}|-|\mathcal{F}_{u+1, n-i+1}|)\ =\ 0.$$
    Hence,
    \begin{align*}
        |\mathcal{F}_{u+1,n+2}|&\ =\ |\mathcal{F}_{u+1,n+1}| + \sum_{i=1}^u (-1)^{i-1}\binom{u}{i}(|\mathcal{F}_{u+1, n-i+2}|-|\mathcal{F}_{u+1, n-i+1}|)\\
        &\ =\ \left(|\mathcal{F}_{u+1, n+1}| + \binom{u}{1}|\mathcal{F}_{u+1, n+1}|\right)+\sum_{i=2}^u (-1)^{i-1}\binom{u}{i}|\mathcal{F}_{u+1, n-i+2}|\\
        &\ - \ \sum_{i=1}^{u-1} (-1)^{i-1}\binom{u}{i}|\mathcal{F}_{u+1, n-i+1}| - (-1)^{u-1}|\mathcal{F}_{u+1, n-u+1}|\\
&\ =\     \binom{u+1}{1}|\mathcal{F}_{u+1, n+1}|  +\sum_{i=2}^u (-1)^{i-1}\binom{u}{i}|\mathcal{F}_{u+1, n-i+2}|\\
        &\ + \ \sum_{i=1}^{u-1} (-1)^{i}\binom{u}{i}|\mathcal{F}_{u+1, n-i+1}| + (-1)^{u}|\mathcal{F}_{u+1, n-u+1}|\\
        &\ =\ \binom{u+1}{1}|\mathcal{F}_{u+1, n+1}|+\sum_{i=2}^u (-1)^{i-1}\binom{u}{i}|\mathcal{F}_{u+1, n-i+2}|\\
        &\ + \ \sum_{i=2}^{u} (-1)^{i-1}\binom{u}{i-1}|\mathcal{F}_{u+1, n-i+2}| + (-1)^{u}|\mathcal{F}_{u+1, n-u+1}|\\
        &\ =\  \binom{u+1}{1}|\mathcal{F}_{u+1, n-1}|+\sum_{i=2}^u (-1)^{i-1}\binom{u+1}{i}|\mathcal{F}_{u+1, n-i+2}|\\
        &\ +\ (-1)^{u}\binom{u+1}{u+1}|\mathcal{F}_{u+1, n-(u+1)+2}|\\
        &\ =\ \sum_{i=1}^{u+1}(-1)^{i-1}\binom{u+1}{i}|\mathcal{F}_{u+1, n-i+2}|.
    \end{align*}
    Therefore, we have shown that for $n\ge 2(u+1)$, 
    $$|\mathcal{F}_{u+1, n}|\ =\ \sum_{i=1}^{u+1}(-1)^{i-1}\binom{u+1}{i}|\mathcal{F}_{u+1, n-i}|.$$
    By mathematical induction, the statement $P(u)$ holds for all $u\ge 1$. 
\end{proof}

%%%%%%%%%%%%%%%%%%%%%%%%%%%%%%%%%%%%%%%%%%%%%%%%%%%%%%%%%%%%%%%%%%%%%%%%%%%%%%%%
%%%%%%%%%%%%%%%%%%%%%%%%%%%%%%%%%%%%%%%%%%%%%%%%%%%%%%%%%%%%%%%%%%%%%%%%%%%%%%%%%
%%%%%%%%%%%%%%%%%%%%%%%%%%%%%%%%%%%%%%%%%%%%%%%%%%%%%%%%%%%%%%%%%%%%%%%%%%%%%%%%%
%%%%%%%%%%%%%%%%%%%%%%%%%%%%%%%%%%%%%%%%%%%%%%%%%%%%%%%%%%%%%%%%%%%%%%%%%%%%%%%%%

\section{Schreier sets that avoid all integers $a$ modulo $b$}\label{sec-amodb}

In this section, we see how the polynomial method described in Section \ref{sec-intro} helps establish complicated linear recurrences whose orders increase as a parameter increases. We also practice the ``divide and conquer" technique by dividing a technical identity into pieces and proving each piece separately. We have used this technique to prove \eqref{e52} by verifying its constituents \eqref{e45}.

We start by finding a formula for $|\mathcal{D}^{(a,b)}_n|$. Each set $F\in \mathcal{D}^{(a,b)}_n$ of size $k$ is uniquely determined by $k-1$ integers from the set
$\{g^{(a,b)}(j_k-1),\ldots, g^{(a,b)}(n-1)\}$, where $j_k$ is the smallest positive integer with  
$g^{(a,b)}(j_k-1)\ge k$. Solving 
$$\left\lfloor\frac{b(j_k-1)-a}{b-1}\right\rfloor + 1\ \ge\ k$$ for $j_k$, we obtain
$$\frac{b(j_k-1)-a}{b-1}\ \ge\ k-1, \mbox{ so }j_k\ =\ k + \left\lceil \frac{a-k+1}{b} \right\rceil.$$
Hence, the number of sets in $\mathcal{D}^{(a,b)}_n$ of size $k$ is
$$\binom{(n-1)-(j_k-1)+1}{k-1}\ =\ \binom{n-j_k+1}{k-1}\ =\ \binom{n-k - \left\lceil \frac{a-k+1}{b}\right\rceil+1}{k-1},$$
which is positive if and only if 
$$n-k - \left\lceil \frac{a-k+1}{b}\right\rceil+1\ \ge\ k-1\mbox{; that is, }k\ \le\ \frac{nb-a}{2b-1}+1.$$
Therefore, 
\begin{equation}\label{e26}
    |\mathcal{D}^{(a,b)}_{n}|\ =\ \sum_{k=1}^{\left\lfloor\frac{nb-a}{2b-1}\right\rfloor+1}\binom{n-k - \left\lceil \frac{a-k+1}{b}\right\rceil+1}{k-1}\ \ =\ \sum_{k=0}^{\left\lfloor\frac{nb-a}{2b-1}\right\rfloor}\binom{n-k + \left\lfloor \frac{k-a}{b}\right\rfloor}{k}.
\end{equation}

For each $b\ge 2$, define the sequence $(d_{b, n})_{n=1}^\infty$ as follows:
$$d_{b,1}\ =\ d_{b, 2}\ =\ \cdots\ =\ d_{b, 2b-1}\ =\ 1$$
and 
$$d_{b, n}\ =\ d_{b, n-b}\  + d_{b, n-2b+1}, \mbox{ for }n\ge 2b.$$
By \cite[Proposition 2]{CGKMTV}, we have the following formula to compute $d_{b,n}$:
\begin{equation}\label{e25}d_{b, n}\ =\ \sum_{i=0}^{\left\lfloor \frac{n-1}{b}\right\rfloor}\binom{\left\lfloor \frac{n+(b-1)i-b}{2b-1}\right\rfloor}{i}, \mbox{ for all }b\ge 2\mbox{ and }n\ge 1.\end{equation}

\begin{prop}\label{kp}
For $a,b\in\mathbb{N}$ with $b\ge 2$ and $1\le a\le b-1$, we have
\begin{equation}\label{e24}|\mathcal{D}^{(a,b)}_n|\ =\ d_{b, nb+1-a}, \mbox{ for all }n\in\mathbb{N}.\end{equation}
\end{prop}

To prove \eqref{e24}, we use \eqref{e25} to write
    \begin{align*}
        d_{b, nb+1-a}&\ =\ \sum_{i=0}^{\left\lfloor \frac{(nb+1-a)-1}{b}\right\rfloor}\binom{\left\lfloor \frac{(nb+1-a)+(b-1)i-b}{2b-1}\right\rfloor}{i}\\
        &\ =\ \sum_{i=0}^{n-1}\binom{\left\lfloor \frac{b(n+1) + (b-1)i-a}{2b-1}\right\rfloor-1}{i}\\
        &\ =\ \sum_{i=0}^{n}\binom{\left\lfloor \frac{b(n+1) + (b-1)i-a}{2b-1}\right\rfloor-1}{i}.
    \end{align*}
    Meanwhile, \eqref{e26} gives
    \begin{align*}
        |\mathcal{D}^{(a,b)}_n|&\ =\ \sum_{k=0}^{\left\lfloor\frac{nb-a}{2b-1}\right\rfloor}\binom{n-k + \left\lfloor \frac{k-a}{b}\right\rfloor}{k}\\
        &\ =\ \sum_{k=0}^{\left\lfloor\frac{nb-a}{2b-1}\right\rfloor}\binom{ \left\lfloor \frac{b(n+1)-(b-1)k-a}{b}\right\rfloor-1}{k}.
    \end{align*}
    Hence, \eqref{e24} is the same as
    $$\sum_{k=0}^{\left\lfloor\frac{nb-a}{2b-1}\right\rfloor}\binom{ \left\lfloor \frac{b(n+1)-(b-1)k-a}{b}\right\rfloor-1}{k}\ =\ \sum_{i=0}^{n}\binom{\left\lfloor \frac{b(n+1) + (b-1)i-a}{2b-1}\right\rfloor-1}{i}, \mbox{ for all }n\ge 1.$$
    Substituting $m = n+1$ gives
    \begin{equation}\label{e23}\sum_{k=0}^{\left\lfloor\frac{(m-1)b-a}{2b-1}\right\rfloor}\underbrace{\binom{ \left\lfloor \frac{bm-(b-1)k-a}{b}\right\rfloor-1}{k}}_{f(k)}\ =\ \sum_{i=0}^{m-1}\underbrace{\binom{\left\lfloor \frac{bm + (b-1)i-a}{2b-1}\right\rfloor-1}{i}}_{g(i)},\mbox{ for all }m\ge 2,\end{equation}
    which generalizes \cite[(6)]{CGKMTV} with the appearance of $a\ge 1$. 

    To see why \eqref{e23} holds for specific cases, which may suggest a proof idea, we collect the terms on both sides of \eqref{e23} when $(m, a, b) = (8, 3, 5)$ and $(11, 2, 7)$:
    \begin{itemize}
        \item Plugging $(m, a, b) = (8, 3, 5)$ into \eqref{e23} gives
        \begin{equation}\label{e80}\underbrace{\binom{6}{0} + \binom{5}{1}}_{ = 6} + \binom{4}{2} + \underbrace{\binom{4}{3}}_{=4}\ =\ \underbrace{\binom{3}{0} + \binom{3}{1}}_{=4} + \binom{4}{2} + \underbrace{\binom{4}{3} + \binom{4}{4} + \binom{5}{5}}_{=6},\end{equation}
        and
        \item Plugging $(m,a,b) = (11, 1, 3)$ into \eqref{e23} gives
        \begin{align}\label{e81}
            \binom{9}{0} &+ \underbrace{\binom{9}{1} + \binom{8}{2}}_{=37} + \binom{7}{3}+ \underbrace{\binom{7}{4} + \binom{6}{5}}_{=41}\nonumber\\
            &\ =\ \underbrace{\binom{5}{0} + \binom{5}{1} + \binom{6}{2} + \binom{6}{3}}_{=41} + \binom{7}{4} + \underbrace{\binom{7}{5} + \binom{7}{6} + \binom{8}{7} + \binom{8}{8}}_{= 37}+ \binom{9}{9}.
        \end{align}
    \end{itemize}
    Equations \eqref{e80} and \eqref{e81} suggest that \eqref{e23} holds because certain terms  on both sides, called posts, are equal, and the remaining terms between every two posts sum up to the same number.
    
    For $0\le i\le m-1$, we consider $i$ of the form
    $$i \ =\ m-(2b-1)(\ell+1)-2a\ =:\ \xi(\ell).$$
    Due to $0\le i\le m-1$, we require
    $$-1\ \le\ \ell\ \le\ \underbrace{\left\lfloor\frac{m-2a}{2b-1}\right\rfloor-1}_{=:\ell_0}.$$
    For $\ell$ in the above range, 
    $$0\ \le\ \chi(\ell)\ :=\ b(\ell+1)+a-1\ \le\ \frac{mb-a}{2b-1}-1\ \le\ \left\lfloor\frac{(m-1)b-a}{2b-1}\right\rfloor.$$
    We have 
    \begin{align}\label{e28}
        g(\xi(\ell))&\ =\ \binom{\left\lfloor\frac{bm+(b-1)(m-(2b-1)(\ell+1)-2a)-a}{2b-1}\right\rfloor-1}{m-(2b-1)(\ell+1)-2a}\nonumber\\
        &\ =\ \binom{m-(b-1)(\ell+1)-a-1}{m-(2b-1)(\ell+1)-2a}\nonumber\\
        &\ =\ \binom{m-(b-1)(\ell+1)-a-1}{b(\ell+1)+a-1}\nonumber\\
        &\ =\ \binom{ \left\lfloor \frac{bm-(b-1)(b(\ell+1)+a-1)-a}{b}\right\rfloor-1}{b(\ell+1)+a-1}\ =\ f(\xi(\ell)).
    \end{align}
    Write 
    \begin{align}\label{e29}
        \sum_{k=0}^{\left\lfloor\frac{(m-1)b-a}{2b-1}\right\rfloor}f(k)&\ =\ \sum_{k=0}^{\chi(-1)-1} f(k) + f(\chi(-1)) + \sum_{k=\chi(-1)+1}^{\chi(0)-1} f(k) + f(\chi(0)) + \cdots\nonumber\\
        &\ +\ \sum_{k=\chi(\ell_0-1)+1}^{\chi(\ell_0)-1} f(k) + f(\chi(\ell_0)) + \sum_{k = \chi(\ell_0)+1}^{\left\lfloor\frac{(m-1)b-a}{2b-1}\right\rfloor}f(k)
    \end{align}
    and 
    \begin{align}\label{e30}
        \sum_{i = 0}^{m-1}g(i)&\ =\ 
        \sum_{i=0}^{\xi(\ell_0)-1}g(i) + g(\xi(\ell_0)) + \sum_{i=\xi(\ell_0)+1}^{\xi(\ell_0-1)-1}g(i) + g(\xi(\ell_0-1)) + \cdots\\
        &\ +\ \sum_{i = \xi(0)+1}^{\xi(-1)-1}g(i) + g(\xi(-1)) + \sum_{i = \xi(-1)+1}^{m-1}g(i).
    \end{align}

Thanks to \eqref{e28}, \eqref{e29}, and \eqref{e30}, to prove \eqref{e23} and thus, Proposition \ref{kp}, it suffices to prove the following three lemmas.

\begin{lem}\label{ki1}
    For $-1\le \ell\le \ell_0-1$, we have
    \begin{equation}\label{e10}
\sum_{k=\chi(\ell)+1}^{\chi(\ell+1)-1}f(k)\ =\ \sum_{i=\xi(\ell+1)+1}^{\xi(\ell)-1}g(i).
\end{equation}
\end{lem}

\begin{lem}\label{ki2} For $m, b\ge 2$ and $1\le a\le b-1$, 
\begin{equation}\label{e11}\sum_{k=0}^{\min\left\{\chi(-1)-1, \left\lfloor\frac{(m-1)b-a}{2b-1}\right\rfloor\right\}} f(k)\ =\ \sum_{i = \max\{\xi(-1)+1,0\}}^{m-1}g(i).\end{equation}    
\end{lem}

\begin{lem}\label{ki3} For $m, b\ge 2$ and $1\le a\le b-1$,
\begin{equation}\label{e12}\sum_{k = \chi(\ell_0)+1}^{\left\lfloor\frac{(m-1)b-a}{2b-1}\right\rfloor}f(k)\ =\ \sum_{i=0}^{\xi(\ell_0)-1}g(i).\end{equation}
\end{lem}

\begin{proof}[Proof of Lemma \ref{ki1}]
The left side has 
\begin{align*}
(\chi(\ell+1)-1)-(\chi(\ell)+1)+1&\ =\ \chi(\ell+1)-\chi(\ell)-1\\
&\ =\ (b(\ell+2)+a-1)-(b(\ell+1)+a-1)-1\\
&\ =\ b-1 \mbox{ terms,}
\end{align*}
while the right side has
\begin{align*}
&(\xi(\ell)-1)-(\xi(\ell+1)+1)+1\\
&\ =\ \xi(\ell)-\xi(\ell+1)-1\\
&\ =\ (m-(2b-1)(\ell+1)-2a)-(m-(2b-1)(\ell+2)-2a)-1\\
&\ =\ 2b-2\mbox{ terms}.
\end{align*}
We use Pascal's rule to combine every two consecutive terms on the right side:
\begin{align*}
    &\sum_{i=\xi(\ell+1)+1}^{\xi(\ell)-1}g(i)\\
    &\ =\ \sum_{i = 1}^{b-1}\left(g(\xi(\ell+1)+2i-1) + g(\xi(\ell+1)+2i)\right)\\ 
    &\ =\ \sum_{i=1}^{b-1} \left(\binom{\left\lfloor \frac{bm + (b-1)(\xi(\ell+1)+2i-1)-a}{2b-1}\right\rfloor-1}{\xi(\ell+1)+2i-1} + \binom{\left\lfloor \frac{bm + (b-1)(\xi(\ell+1)+2i)-a}{2b-1}\right\rfloor-1}{\xi(\ell+1)+2i}\right)\\
    &\ =\ \sum_{i=1}^{b-1}\binom{m-(b-1)(\ell+2)-a+i-2}{\xi(\ell+1)+2i-1}\\
    &\ +\ \sum_{i=1}^{b-1}\binom{m-(b-1)(\ell+2)-a+i-2}{\xi(\ell+1)+2i}\\
    &\ =\ \sum_{i=1}^{b-1}\binom{m-(b-1)(\ell+2)-a+i-1}{m-(2b-1)(\ell+2)-2a+2i}\\
    &\ =\ \sum_{i=1}^{b-1}\binom{m-(b-1)(\ell+2)-a+i-1}{b(\ell+2)+a-i-1}.
\end{align*}
Similarly, 
\begin{align*}
    \sum_{k=\chi(\ell)+1}^{\chi(\ell+1)-1}f(k)&\ =\ \sum_{k=1}^{b-1}f(\chi(\ell+1)-k)\\
    &\ =\ \sum_{k=1}^{b-1}\binom{ \left\lfloor \frac{bm-(b-1)(\chi(\ell+1)-k)-a}{b}\right\rfloor-1}{\chi(\ell+1)-k}\\
    &\ =\ \sum_{k=1}^{b-1}\binom{ \left\lfloor \frac{bm-(b-1)(b(\ell+2)+a-1-k)-a}{b}\right\rfloor-1}{b(\ell+2)+a-1-k}\\
    &\ =\ \sum_{k=1}^{b-1}\binom{m-(b-1)(\ell+2)-a+k-1}{b(\ell+2)+a-k-1}.
\end{align*}
Therefore, \eqref{e10} holds. 
\end{proof}

\begin{proof}[Proof of Lemma \ref{ki2}]
We need to prove 
\begin{equation}\label{e13}\sum_{k=0}^{\min\left(\left\lfloor\frac{(m-1)b-a}{2b-1}\right\rfloor, a-2\right)}\binom{\left\lfloor\frac{bm-(b-1)k-a}{b}\right\rfloor-1}{k}\ =\ \sum_{i = \max\{m-2a+1,0\}}^{m-1}\binom{\left\lfloor \frac{bm+(b-1)i-a}{2b-1}\right\rfloor-1}{i}.\end{equation}

\bigskip

\noindent Case 1: $m\le 2a-2$. Then 
$$\left\lfloor\frac{(m-1)b-a}{2b-1}\right\rfloor\ \le\ \left\lfloor\frac{(2a-3)b-a}{2b-1}\right\rfloor\ =\ \left\lfloor a-\frac{3b}{2b-1}\right\rfloor\ =\ a-2.$$
Hence, \eqref{e13} becomes
\begin{equation}\label{e14}\sum_{k=0}^{\left\lfloor\frac{(m-1)b-a}{2b-1}\right\rfloor}\binom{\left\lfloor\frac{bm-(b-1)k-a}{b}\right\rfloor-1}{k}\ =\ \sum_{i = 0}^{m-1}\binom{\left\lfloor \frac{bm+(b-1)i-a}{2b-1}\right\rfloor-1}{i}.\end{equation}
We have
\begin{align*}
    \sum_{k=0}^{\left\lfloor\frac{(m-1)b-a}{2b-1}\right\rfloor}\binom{\left\lfloor\frac{bm-(b-1)k-a}{b}\right\rfloor-1}{k}
    &\ =\ \sum_{k=0}^{\left\lfloor\frac{(m-1)b-a}{2b-1}\right\rfloor}\binom{ m - k + \left\lfloor\frac{k-a}{b}\right\rfloor-1}{k}\\
    &\ =\ \sum_{k=0}^{\left\lfloor\frac{(m-1)b-a}{2b-1}\right\rfloor}\binom{ m - k -2}{k}
\end{align*}
because
$$-1 \ <\ \frac{-a}{b}\ \le\ \frac{k-a}{b}\ \le\ \frac{\frac{(m-1)b-a}{2b-1}-a}{b}\ =\ \frac{m-1-2a}{2b-1}\ \le\ \frac{2a-2-1-2a}{2b-1}\ =\ \frac{-3}{2b-1},$$
and thus, 
$\left\lfloor\frac{k-a}{b}\right\rfloor = -1$. Therefore, \eqref{e14} is the same as
\begin{equation}\label{e15}\sum_{k=0}^{\left\lfloor\frac{(m-1)b-a}{2b-1}\right\rfloor}\binom{ m - k -2}{k}\ =\ \sum_{i = 0}^{m-1}\binom{\left\lfloor \frac{bm+(b-1)i-a}{2b-1}\right\rfloor-1}{i}.\end{equation}

The following proof of \eqref{e15} is written by ChatGPT, OpenAI's GPT-5.5 model \cite{GPT}, which calls for the next lemma. For a proof of the lemma, see Appendix \ref{ff_lem}.

\begin{lem}\label{GPTlem}
Given $b\ge 2$, $1\le a\le b-1$, $2\le m\le 2a-2$, and $1\le t\le m$, we have
$$\left\lfloor\frac{bt-a}{2b-1}\right\rfloor\ =\ \left\lfloor\frac{t-1}{2}\right\rfloor.$$
\end{lem}

By Lemma \ref{GPTlem}, the left side of \eqref{e15} is equal to
\begin{equation}\label{e16}\sum_{k=0}^{\left\lfloor\frac{(m-1)b-a}{2b-1}\right\rfloor}\binom{ m - k -2}{k}\ =\ \sum_{k=0}^{\left\lfloor\frac{m-2}{2}\right\rfloor}\binom{ m - k -2}{k}\ =\ \sum_{k\ge 0}\binom{m-k-2}{k}.\end{equation}
Meanwhile, the right side of \eqref{e15} is equal to
\begin{align}\label{e17}
    &\sum_{i = 0}^{m-1}\binom{\left\lfloor \frac{bm+(b-1)i-a}{2b-1}\right\rfloor-1}{i}\nonumber\\
    &\ =\ \sum_{i=1}^m \binom{\left\lfloor \frac{bm+(b-1)(m-i)-a}{2b-1}\right\rfloor-1}{m-i}\nonumber\\
    &\ =\  \sum_{i=1}^m \binom{m-i-1+\left\lfloor \frac{bi-a}{2b-1}\right\rfloor}{m-i}\nonumber\\
    &\ =\ \sum_{i = 1}^m \binom{m-i-1+\left\lfloor \frac{i-1}{2}\right\rfloor}{m-i}\nonumber\\
    &\ = \ \sum_{i= 1}^{\lfloor m/2\rfloor}\binom{m-2i-1+\left\lfloor \frac{2i-1}{2}\right\rfloor}{m-2i} + \sum_{i = 0}^{\lfloor (m-1)/2\rfloor}\binom{m-(2i+1)-1+\left\lfloor \frac{(2i+1)-1}{2}\right\rfloor}{m-(2i+1)}\nonumber\\
     &\ = \ \sum_{i = 1}^{\lfloor m/2\rfloor}\binom{m-i-2}{m-2i} + \sum_{i = 0}^{\lfloor (m-1)/2\rfloor}\binom{m-i-2}{m-2i-1}\nonumber\\
     &\ =\ \begin{cases}
\sum_{i=1}^p \binom{2p-i-2}{2p-2i} + \sum_{i=0}^{p-1} \binom{2p-i-2}{2p-2i-1}, &\mbox{ if }m=2p, p\ge 1,\\
\sum_{i=1}^p\binom{2p-i-1}{2p+1-2i} +\sum_{i=0}^{p}\binom{2p-i-1}{2p-2i}, &\mbox{ if }m=2p+1, p\ge 1,
     \end{cases}\nonumber\\
     &\ =\ \begin{cases}
\sum_{i=1}^{p-1} \binom{2p-i-1}{2p-2i} + \binom{p-2}{0}, &\mbox{ if }m=2p, p\ge 1,\\
\sum_{i=1}^p\binom{2p-i}{2p+1-2i}, &\mbox{ if }m=2p+1, p\ge 1,
     \end{cases}\nonumber\\
     &\ =\ \begin{cases}
\sum_{i=1}^{p} \binom{2p-i-1}{2p-2i}, &\mbox{ if }m=2p, p\ge 1,\\
\sum_{i=1}^p\binom{2p-i}{2p+1-2i}, &\mbox{ if }m=2p+1, p\ge 1,
     \end{cases}\nonumber\\
     &\ =\ \sum_{i=1}^{\lfloor m/2\rfloor}\binom{m-i-1}{m-2i}\ =\ \sum_{i=1}^{\lfloor m/2\rfloor}\binom{m-i-1}{i-1}\ =\ \sum_{i=0}^{\lfloor (m-2)/2\rfloor} \binom{m-k-2}{k}.
\end{align}
By \eqref{e16} and \eqref{e17}, we obtain \eqref{e15}.

\bigskip

Case 2: $m\ge 2a-1$.  Then \eqref{e13} becomes
\begin{equation}\label{e18}
\sum_{k=0}^{a-2}\binom{\left\lfloor\frac{bm-(b-1)k-a}{b}\right\rfloor-1}{k}\ =\ \sum_{i =m-2a+1}^{m-1}\binom{\left\lfloor \frac{bm+(b-1)i-a}{2b-1}\right\rfloor-1}{i}.
\end{equation}
We have
\begin{equation}\label{e19}
   \sum_{k=0}^{a-2}\binom{ \left\lfloor \frac{bm-(b-1)k-a}{b}\right\rfloor-1}{k}\ =\ \sum_{k=0}^{a-2}\binom{m-k + \left\lfloor \frac{k-a}{b}\right\rfloor-1}{k}\ =\ \sum_{k=0}^{a-2}\binom{m-k -2}{k}. 
\end{equation}
Similarly, 
\begin{align}\label{e20}
    &\sum_{i = m-2a+1}^{m-1}\binom{\left\lfloor \frac{bm + (b-1)i-a}{2b-1}\right\rfloor-1}{i}\ =\ \sum_{i = m-2a+1}^{m-2}\binom{\left\lfloor \frac{bm + (b-1)i-a}{2b-1}\right\rfloor-1}{i}\nonumber\\
    &\ =\ \sum_{i=1}^{a-1}\left(\binom{\left\lfloor \frac{bm + (b-1)(m-2a-1+2i)-a}{2b-1}\right\rfloor-1}{m-2a-1+2i} + \binom{\left\lfloor \frac{bm + (b-1)(m-2a+2i)-a}{2b-1}\right\rfloor-1}{m-2a+2i}\right)\nonumber\\
      &\ =\ \sum_{i=1}^{a-1}\left(\binom{m-a+i-2+\left\lfloor\frac{b-i}{2b-1}\right\rfloor}{m-2a-1+2i} + \binom{ m-a+i-1+\left\lfloor \frac{-i}{2b-1}\right\rfloor}{m-2a+2i}\right)\nonumber\\
      &\ =\ \sum_{i=1}^{a-1}\left(\binom{m-a+i-2}{m-2a-1+2i}+\binom{ m-a+i-2}{m-2a+2i}\right)\nonumber\\
      &\ =\  \sum_{i=1}^{a-1}\binom{m-a+i-1}{m-2a+2i}\nonumber\\
      &\ =\ \sum_{i=1}^{a-1}\binom{m-a+i-1}{a-i-1}\nonumber\\
      &\ =\ \sum_{k=0}^{a-2}\binom{m-k-2}{k}.
\end{align}
Then \eqref{e18} follows immediately from \eqref{e19} and \eqref{e20}. 
\end{proof}

\begin{proof}[Proof of Lemma \ref{ki3}]
    We write \eqref{e12} out as 
    \begin{equation}\label{e21}\sum_{k=b\left\lfloor \frac{m-2a}{2b-1}\right\rfloor+a}^{\lfloor \frac{(m-1)b-a}{2b-1}\rfloor}\binom{\left\lfloor\frac{bm-(b-1)k-a}{b}\right\rfloor-1}{k}\ =\ \sum_{i=0}^{m-(2b-1)\left\lfloor \frac{m-2a}{2b-1}\right\rfloor-2a-1}\binom{\left\lfloor \frac{bm+(b-1)i-a}{2b-1}\right\rfloor-1}{i}.\end{equation}
    Let $m = (2b-1)s+t$ with $s\ge 0$ and $0\le t \le 2b-2$. We assume $t = 2r+1$ with $r\ge 0$; the proof for even $t$ is similar. Then
    \begin{align*}&\left\lfloor \frac{(m-1)b-a}{2b-1}\right\rfloor\ =\ \left\lfloor \frac{((2b-1)s+t-1)b-a}{2b-1}\right\rfloor\ =\ bs + \left\lfloor\frac{bt-b-a}{2b-1}\right\rfloor\\
    &\ = \ bs + \left\lfloor\frac{b(2r+1)-b-a}{2b-1}\right\rfloor\ =\ bs + r + \left\lfloor\frac{r-a}{2b-1}\right\rfloor,\\
    &b\left\lfloor\frac{m-2a}{2b-1}\right\rfloor+a\ =\ b\left\lfloor\frac{(2b-1)s+t-2a}{2b-1}\right\rfloor+a\ =\ bs + b\left\lfloor \frac{t-2a}{2b-1}\right\rfloor + a,
    \mbox{ and }\\
    &m-(2b-1)\left\lfloor \frac{m-2a}{2b-1}\right\rfloor-2a-1\\
    &\ =\ m-(2b-1)s-(2b-1)\left\lfloor\frac{t-2a}{2b-1}\right\rfloor-2a-1\\
    &\ = \ t-(2b-1)\left\lfloor \frac{t-2a}{2b-1}\right\rfloor -2a -1.
    \end{align*}

\bigskip

\noindent Case 1: $t\le 2a-1$. We get
$\lfloor(t-2a)/(2b-1)\rfloor = -1$. The right side of \eqref{e21} is
\begin{align*}
    &\sum_{i=0}^{2r+2b-2a-1}\binom{bs+\left\lfloor\frac{b(2r+1)+(b-1)i-a}{2b-1}\right\rfloor-1}{i}\\
    &\ =\ \sum_{j=0}^{r+b-a-1}\left(\binom{bs+\left\lfloor\frac{b(2r+1)+(b-1)2j-a}{2b-1}\right\rfloor-1}{2j} + \binom{bs+\left\lfloor\frac{b(2r+1)+(b-1)(2j+1)-a}{2b-1}\right\rfloor-1}{2j+1}\right)\\
    &\ =\ \sum_{j=0}^{r+b-a-1}\left(\binom{bs+r+j+\left\lfloor\frac{r+b-j-a}{2b-1}\right\rfloor-1}{2j} + \binom{bs+r+j+\left\lfloor \frac{r-j-a}{2b-1}\right\rfloor}{2j+1}\right)\\
    &\ =\ \sum_{j=0}^{r+b-a-1}\left(\binom{bs+r+j-1}{2j} + \binom{bs+r+j-1}{2j+1}\right)\\
    &\ =\ \sum_{j=0}^{r+b-a-1}\binom{bs+r+j}{2j+1},
\end{align*}
and the left side is
\begin{align*}
    \sum_{k=bs-b+a}^{bs+r-1}\binom{\left\lfloor\frac{bm-(b-1)k-a}{b}\right\rfloor-1}{k}&\ =\ \sum_{k=bs-b+a}^{bs+r-1}\binom{(2b-1)s+t-k+\lfloor\frac{k-a}{b}\rfloor-1}{k}\\
    &\ =\ \sum_{k=bs-b+a}^{bs+r-1}\binom{2bs+t-k-2}{k}\\
    &\ =\ \sum_{k=bs-b+a}^{bs+r-1}\binom{2bs+2r-k-1}{k}\\
    &\ =\ \sum_{k=bs-b+a}^{bs+r-1}\binom{2bs+2r-k-1}{2bs+2r-2k-1}\\
    &\ =\ \sum_{j=0}^{r+b-a-1}\binom{bs+r+j}{2j+1}.
\end{align*}
Hence, \eqref{e21} holds in this case. 

\bigskip

\noindent Case 2: $t\ge 2a$. We get
$\lfloor(t-2a)/(2b-1)\rfloor = 0$. The right side of \eqref{e21} is
\begin{align*}
    &\sum_{i=0}^{2r-2a}\binom{bs+\left\lfloor\frac{b(2r+1)+(b-1)i-a}{2b-1}\right\rfloor-1}{i}\\
    &\ =\ \sum_{j=0}^{r-a-1}\left(\binom{bs+\left\lfloor\frac{b(2r+1)+(b-1)2j-a}{2b-1}\right\rfloor-1}{2j} + \binom{bs+\left\lfloor\frac{b(2r+1)+(b-1)(2j+1)-a}{2b-1}\right\rfloor-1}{2j+1}\right)\\
    &\ +\ \binom{bs+2r-a-1}{2r-2a}\\
    &\ =\ \sum_{j=0}^{r-a-1}\left(\binom{bs+r+j+\left\lfloor\frac{r+b-j-a}{2b-1}\right\rfloor-1}{2j} + \binom{bs+r+j+\left\lfloor \frac{r-j-a}{2b-1}\right\rfloor}{2j+1}\right)\\
    &\ +\ \binom{bs+2r-a-1}{2r-2a}\\
    &\ =\ \sum_{j=0}^{r-a-1}\binom{bs+r+j-1}{2j} + \sum_{j=0}^{r-a-1}\binom{bs+r+j}{2j+1}+ \binom{bs+2r-a-1}{2r-2a}\\
    &\ =\ \binom{bs+r-1}{0}+\sum_{j=1}^{r-a-1}\binom{bs+r+j-1}{2j} + \sum_{j=0}^{r-a-2}\binom{bs+r+j}{2j+1}\\
    &\ +\ \binom{bs+2r-a-1}{2r-2a-1} + \binom{bs+2r-a-1}{2r-2a}\\
    &\ =\ 1 + \sum_{j=0}^{r-a-2}\left(\binom{bs+r+j}{2j+2} + \binom{bs+r+j}{2j+1}\right) + \binom{bs+2r-a}{2r-2a}\\
    &\ =\ 1 + \sum_{j=0}^{r-a-2}\binom{bs+r+j+1}{2j+2} + \binom{bs+2r-a}{2r-2a}\ =\ \sum_{j=-1}^{r-a-1}\binom{bs+r+j+1}{2j+2},
\end{align*}
and the left side is
\begin{align*}
    &\sum_{k=bs+a}^{bs+r}\binom{(2b-1)s+t-k+\left\lfloor\frac{k-a}{b}\right\rfloor-1}{k}\\
    &\ =\ \sum_{j=a}^{r}\binom{2bs+t-bs-j-1}{bs+j}\\
    &\ =\ \sum_{j=a}^{r}\binom{bs+2r-j}{bs+j}\\
    &\ =\ \sum_{j=a}^r\binom{bs+2r-j}{2r-2j}\ =\ \sum_{i=-1}^{r-a-1}\binom{bs+r+i+1}{2i+2}, \mbox{ with }i = r-j-1.
\end{align*}
We have, therefore, confirmed \eqref{e21} in this case. 
\end{proof}

We have shown that $(|\mathcal{D}^{(a,b)}_n|)_{n=1}^\infty$ is a $b$-periodic subsequence of $(d_{b,n})_{n=1}^\infty$, which satisfies the relatively simple recurrence $d_{b, n} = d_{b, n-b}  + d_{b, n-2b+1}$, for $n\ge 2b$. To deduce a linear recurrence for $(|\mathcal{D}^{(a,b)}_n|)_{n=1}^\infty$ from its parent sequence $(d_{b,n})_{n=1}^\infty$, we need to encode linear recurrences by polynomials. 

\begin{defi}\normalfont
    Let $p(x) = c_0+c_1x + \cdots +c_mx^m$ be a polynomial with real coefficients $(c_i)_{i=1}^m$. A sequence $(a_n)_{n=1}^\infty$ is said to satisfy $p(x)$ if
$$c_0a_n + c_1a_{n-1} + \cdots + c_ma_{n-m} \ =\ 0, \mbox{ for all }n \ge m + 1.$$
\end{defi}

\begin{rek}\label{rek1}\normalfont
   Suppose that a polynomial $p(x)$ divides a polynomial $q(x)$. If the sequence $(a_n)_{n=1}^\infty$ satisfies $p(x)$, then $(a_n)_{n=1}^\infty$ satisfies $q(x)$. 
\end{rek}

\begin{proof}[Proof of Theorem \ref{mt3}]
    We want to show that $(|\mathcal{D}^{(a,b)}_n|)_{n=1}^\infty$
    satisfies the recurrence
    $$|\mathcal{D}^{(a,b)}_n|\ =\ \sum_{i=1}^b(-1)^{i-1}\binom{b}{i}|\mathcal{D}^{(a,b)}_{n-i}| + |\mathcal{D}^{(a,b)}_{n-2b+1}|$$
    or equivalently, the polynomial
    \begin{align*}
    q(x)&\ :=\ 1 - \sum_{i=1}^b (-1)^{i-1}\binom{b}{i}x^i - x^{2b-1}\\
    &\ =\ \sum_{i=0}^b (-1)^{i}\binom{b}{i}x^i - x^{2b-1}\\
    &\ =\ \sum_{i=0}^b \binom{b}{i}(-x)^i - x^{2b-1}\ =\ (1-x)^b - x^{2b-1}.
    \end{align*}
    Meanwhile, the sequence $(d_{b,n})_{n=1}^\infty$ satisfies 
    $$p(x)\ :=\ 1 - x^b-x^{2b-1}.$$
    We have
    \begin{align*}
        q(x^b) &\ =\  (1-x^b)^b - (x^{2b-1})^b\\
        &\ =\ (1-x^b - x^{2b-1})\sum_{i=1}^{b-1}(1-x^b)^i (x^{2b-1})^{b-i}\\
        &\ =\ p(x)\sum_{i=1}^{b-1}(1-x^b)^i (x^{2b-1})^{b-i}.
    \end{align*}
    Hence, the polynomial $p(x)$ divides the polynomial $q(x^b)$. By Remark \ref{rek1}, the sequence $(d_{b,n})_{n=1}^\infty$ satisfies $q(x^b)$. Since $(|\mathcal{D}^{(a,b)}_n|)_{n=1}^\infty$ is a $b$-periodic subsequence of $(d_{b,n})_{n=1}^\infty$, it follows that $(|\mathcal{D}^{(a,b)}_n|)_{n=1}^\infty$ satisfies $q(x)$. 
\end{proof}

\appendix
\section{Equivalence between two definitions of symmetric sets}\label{symset}
\begin{proof}[Proof that \eqref{e1} is equivalent to \eqref{e2}]
    Assume \eqref{e1}. Pick $m\in s_A-\left\{a\in A\,:\, a < s_A\right\}$. Then there exists $a'\in A$ with $a' < s_A$ and  $m = s_A -a'$. By \eqref{e1}, there exists $a''\in A$ with $2s_A - a'' = a'$. It follows that $m = s_A - (2s_A - a'') = a'' - s_A$. Since $m > 0$, we have $a'' \in \left\{a\in A\,:\, a > s_A\right\}$, so $m\in \left\{a\in A\,:\, a > s_A\right\}-s_A$. We have shown that $s_A-\left\{a\in A\,:\, a < s_A\right\} \subset \left\{a\in A\,:\, a > s_A\right\}-s_A$. By symmetry, we obtain
    $s_A-\left\{a\in A\,:\, a < s_A\right\} = \left\{a\in A\,:\, a > s_A\right\}-s_A$.

    Assume \eqref{e2}. Pick $a^*\in A$. Without loss of generality, assume that 
    $a^* < s_A$. By \eqref{e2}, there exists $a^{**}\in A$ with $a^{**} > s_A$ and $a^{**}-s_A = s_A - a^*$, so $a^* = 2s_A - a^{**}$, which is in $2s_A - A$. Hence, $A\subset 2s_A - A$. By symmetry, we obtain $A = 2s_A - A$.
\end{proof}

\section{Bijectivity of $\xi_{n,k}$}\label{bi-xi}

\begin{lem}
   For $n\ge 2$ and even $k\ge 6$, the map $\xi_{n,k}$ is a bijection. As a result, 
$$|\mathcal{M}_{2n+3}(k+1)^{(1, \ge 3)}| \ =\ |\mathcal{M}_{2n+5}(k+2)^{(1, = 2), \left(\frac{k}{2}, \ge 2\right)}|.$$ 
\end{lem}

\begin{proof}
Recall the map $\xi_{n,k}: \mathcal{M}_{2n+3}(k+1)^{(1, \ge 3)} \rightarrow \mathcal{M}_{2n+5}(k+2)^{(1, = 2), \left(\frac{k}{2}, \ge 2\right)}$,
\begin{align*}
    &F\ \mapsto\ \{\min F+1, \min F+3\}\cup(\{a\in F\,:\,\min_2 F\le a\le  \min_{k/2-1} F\}+1)\\
    &\cup\{\min_{k/2}F+2, \max_{k/2} F+1\}\\
    &\cup(\{a\in F\,:\,\max_{k/2-1} F\le a\le  \max_{2} F\}+2)\cup\{\max F, \max F+2\}.
\end{align*}
We have
\begin{align*}
    &\min_2 \xi_{n,k}(F) - \min\xi_{n,k}(F)\ =\ (\min F + 3)-(\min F+1)\ = \ 2,\\
    &\min_{k/2+1}\xi_{n,k}(F) - \min_{k/2}\xi_{n,k}(F)\ =\ (\min_{k/2} F + 2) - (\min_{k/2-1} F+1)\ \ge\ 2,\\
    &\max \xi_{n,k}(F)\ =\ \max F+2\ =\ (2n+3)+2\ =\ 2n+5,\\
    &\min \xi_{n,k}(F)\ =\ \min F + 1\ =\ k+2\ =\ |F|+1\ =\ |\xi_{n,k}(F)|,
\end{align*}
and $\xi_{n,k}(F)$ is symmetric because $F$ is symmetric.

The definition of $\xi_{n,k}$ guarantees injectivity. We prove surjectivity of $\xi_{n,k}$. Pick $E\in \mathcal{M}_{2n+5}(k+2)^{(1, = 2), \left(\frac{k}{2}, \ge 2\right)}$. Let
\begin{align*}
F \ :=\ \{\min E-1\}&\cup(\{a\in E\,:\,\min_3 E\le a\le  \min_{k/2} E\}-1)\\
&\cup \{\min_{k/2+1} E - 2, n+k/2+2, \max_{k/2+1}E-1\}\\
&\cup(\{a\in E\,:\,\max_{k/2} E\le a\le  \max_{3} E\}-2)\cup\{\max E - 2\}.
\end{align*}
Note that 
$$\min_{k/2+1} E\ <\ \frac{(2n+5)+(k+2)}{2}\ <\ \max_{k/2+1} E,$$
so
$$\min_{k/2+1} E-2\ <\ n+\frac{k}{2}+2\ <\ \max_{k/2+1} E-1.$$
We have
\begin{align*}
    &\min_2 F-\min F\ =\ (\min_3 E - 1)- (\min E-1)\ =\ \min_3 E - \min E\ \ge\ 3, \\
    &\max F\ =\ \max E - 2\ =\ (2n+5) - 2\ =\ 2n+3,\\
    &\min F\ =\ \min E - 1\ =\ k+1\ =\ |E| - 1\ =\ |F|,
\end{align*}
and $F$ is symmetric because $E$ is symmetric. Hence, the set $F$ is in $\mathcal{M}_{2n+3}(k+1)^{(1,\ge 3)}$. Furthermore, we have $\xi_{n,k}(F) = E$, so $\xi_{n,k}$ is surjective. 
\end{proof}

\section{A floor-function lemma}\label{ff_lem}
\begin{proof}[Proof of Lemma \ref{GPTlem}]
    If $t = 2s$, then $s+1\le a$, and
    $$(s-1)(2b-1)\ \le\ 2bs-a\ <\ (2b-1)s.$$
    Hence, $$s-1\ \le\ \frac{2bs-a}{2b-1}\ <\ s,
    \mbox{ which gives }\left\lfloor \frac{tb-a}{2b-1}\right\rfloor\ =\ \left\lfloor \frac{2bs-a}{2b-1}\right\rfloor\ =\ s-1\ =\ \left\lfloor\frac{t-1}{2}\right\rfloor.$$
    If $t = 2s+1$, then $s+2\le a$, and 
    $$s(2b-1)\ \le\ (2s+1)b-a\ <\ (s+1)(2b-1),$$
    which gives
    $$s\ \le\ \frac{(2s+1)b-a}{2b-1}\ <\ s+1\mbox{, and }\left\lfloor\frac{tb-a}{2b-1}\right\rfloor\ =\ \left\lfloor \frac{(2s+1)b-a}{2b-1}\right\rfloor\ =\ s\ =\ \left\lfloor\frac{t-1}{2}\right\rfloor.$$
\end{proof}
%%%%%%%%%%%%%%%%%%%%%%%%%%%%%%%%%%%%%%%%%%%%%%%%%%%%%%%%%%%%%%%%%%%%%%%%%%%%%%%%%%%%%%%%%%%%%%%%%%%%%%%%%%%%%%%%%%%%%%%%%%%%%%%%%%%%%%%%%%%%%%%%%%%%%%%%%%%%%%%%%%%%%%%%%%%%%%%%%%%%%%%%%%%%%%%%%%%%%%%%%%%%%%%%%%%%%%%%%%%%%%%%%%%%%%%%%%%%%%%%%%%%%%%%%%%%%%%%%%%%%%%%%%%%%%%%%%%%%%%%%%%%%%%%%%%%%%%%%%%%%%%%%%%%%%%%%%%%%%%%%%%%%%%%%%%%%%%%%%%%%%%%%%%%%%%%%%%%%%%%%%%%%%%%%%%%

\ \\

\begin{thebibliography}{9}
\bibitem{BC} K. Beanland and H. V. Chu, On Schreier-type sets, partitions,
and compositions, \textit{J. of Integer Seq.} \textbf{27} (2024), 13 pp.
\bibitem{BCF} K. Beanland, H. V. Chu, and C. E. Finch-Smith, Generalized Schreier sets, linear recurrence relation, and Tur\'{a}n graphs, \textit{Fibonacci Quart.} \textbf{60} (2022), 352--356.
\bibitem{BGHH} K. Beanland, D. Gorovoy, J. Hodor, and D. Homza, Counting unions of Schreier sets, \textit{Bull. Aust. Math. Soc.} \textbf{110} (2024), 19--31.
\bibitem{BQ} A. T. Benjamin and J. J. Quinn, \textit{Proofs That Really Count: The Art of Combinatorial Proof}, Mathematical Association of America, Washington, DC, 2003.
\bibitem{Bi} A. Bird, Jozef Schreier, Schreier sets, and the Fibonacci sequence, 2012. Available at \url{https://outofthenormmaths.wordpress.com/2012/05/13/jozef-schreier-schreier-sets-and-the-fibonacci-sequence/}.
\bibitem{C1} H. V. Chu, The Fibonacci sequence and Schreier-Zeckendorf sets, \textit{J. of Integer Seq.} \textbf{22} (2019), 12 pp.
\bibitem{C2} H. V. Chu, Partial sums of the Fibonacci sequence, \textit{Fibonacci Quart.} \textbf{59} (2021), 132--135.
\bibitem{C3} H. V. Chu, Various sequences from counting subsets, \textit{Fibonacci Quart.} \textbf{59} (2021), 152--157.

\bibitem{CGKMTV} H. V. Chu, Y. Geng, J. King, S. J. Miller, G. Tresch, and Z. L. Vasseur, Linear recurrences from counting Schreier-type multisets, \textit{Integers} \textbf{26} (2026), 24 pp. 
\bibitem{CIMSZ} H. V. Chu, N. Irmark, S. J. Miller, L. Szalay, and S. X. Zhang, 
Schreier multisets and the s-step Fibonacci sequences, \textit{Integers} \textbf{24A} (2024), 11 pp. 
\bibitem{CMX} H. V. Chu, S. J. Miller, and Z. Xiang, Higher order Fibonacci sequences from generalized Schreier sets, \textit{Fibonacci Quart.}  \textbf{58} (2020), 249--253. 
\bibitem{CV0} H. V. Chu and Z. L. Vasseur, Weighted Schreier-type sets and the Fibonacci sequence, \textit{Fibonacci Quart.} \textbf{62} (2024), 305--315. 
\bibitem{CV} H. V. Chu and Z. L. Vasseur, Linear recurrences of generalized Schreier sets revisited, \textit{J. of Integer Seq.} \textbf{29} (2026), 18 pp.
\bibitem{KKMW} M. Kolo\v{g}lu, G. S. Kopp, S. J. Miller, and Y. Wang, On the number of summands in Zeckendorf decompositions, \textit{Fibonacci Quart.} \textbf{49}, 116--130.
\bibitem{GPT} OpenAI. (2026). ChatGPT (GPT-5.5) [Large language model]. \url{https://chatgpt.com}.
\end{thebibliography}
\end{document}